\makeatletter
\def\Mapstofill@{\arrowfill@{\Mapstochar\Relbar}\Relbar\Rightarrow}
\makeatother

\documentclass[oneside,11pt]{amsart}
\usepackage[marginpar=2.5cm]{geometry}
\usepackage{soul}
\usepackage[utf8]{inputenc}
\usepackage{etoolbox}
\usepackage{thm-restate}
\usepackage{comment}
\usepackage{todonotes}
\usepackage[colorlinks=false]{hyperref}
\usepackage{centernot}
\usepackage{mathtools}
\usepackage{amsmath,amssymb,amsthm, mathrsfs, mathtools, bbold, mathbbol}
\usepackage{amscd,mathtools}
\usepackage{stmaryrd}
\usepackage{microtype}

 \newcommand{\gothic}{\mathfrak}
\usepackage{amsmath}
\usepackage{amssymb}
\usepackage{amsfonts}
\usepackage{soul}
\usepackage{amsthm}
\usepackage{graphicx}
\usepackage{tikz,tikz-cd}
\usepackage{color}
\usepackage{mathtools}
\usepackage{enumitem}
\usepackage{hyperref}

\newtheorem{thm}{Theorem}[section]
\newtheorem{theorem}{Theorem}[section]
\newtheorem{prop}[thm]{Proposition}
\newtheorem{proposition}[thm]{Proposition}
\newtheorem{lem}[thm]{Lemma}
\newtheorem{lemma}[thm]{Lemma}
\newtheorem{cor}[thm]{Corollary}

\theoremstyle{definition}
\newtheorem{defn}[thm]{Definition}
\newtheorem{rem}[thm]{Remark}

\newtheorem{exmp}[thm]{Example}

  \newtheorem{introthm}{Theorem}
  \newtheorem{introcor}[introthm]{Corollary}

\newcommand{\abs}[1]{\lvert{#1}\rvert}

\newcommand{\norm}[1]{\left\lVert{#1}\right\rVert}
\renewcommand{\bar}[1]{\overline{#1}}

\newcommand{\set}[2]{\{\,{#1} \mid{#2} \,\}}
\renewcommand{\emptyset}{\varnothing}

\renewcommand{\setminus}{-}

\newcommand{\field}[1]{\mathbb{#1}}
\newcommand{\Z}{\field{Z}}
\newcommand{\R}{\field{R}}

\newcommand{\calA}{\mathcal{A}}
  \newcommand{\calB}{\mathcal{B}}

  \newcommand{\calP}{\mathcal{P}}

  \newcommand{\calU}{\mathcal{U}}
    \newcommand{\calV}{\mathcal{V}}

\newcommand{\cG}{\mathcal{G}}

\newcommand{\cP}{\mathcal{P}}

\newcommand{\qq}{\mathfrak{q}}
\newcommand{\go}{\mathfrak{o}}
\newcommand{\bfa}{\textbf{a}}
\newcommand{\bfb}{\textbf{b}}

\newcommand{\Ga}{\mathfrak{a}}
\newcommand{\Gb}{\mathfrak{b}}
\newcommand{\Gc}{\mathfrak{c}}
\newcommand{\ua}{\underline{a}}

\DeclareMathOperator{\Sh}{Sh}

\DeclareMathOperator{\CAT}{CAT}

\DeclareMathOperator{\Cay}{Cay}

\DeclarePairedDelimiterX{\Norm}[1]{\lVert}{\rVert}{#1}

\DeclareMathOperator{\Sat}{Sat} %

\DeclareMathOperator{\deep}{deep}
\DeclareMathOperator{\trans}{trans}

\usepackage{ifthen}

\newsavebox{\commentbox}

\newcounter{acomments}

\newcounter{hcomments}

\usepackage{import}
\usepackage{xifthen}
\usepackage{pdfpages}
\usepackage{transparent}

  \author{Hoang Thanh Nguyen}
 \address{Department of Mathematics, FPT University, DaNang, VietNam}
 \email{nthoang.math@gmail.com}
 
  \author{Yulan Qing}
 \address{Department of Mathematics,  University of Tennessee at Knoxville, Knoxville, Tennessee, USA}
 \email{yqing@utk.edu}

\begin{document}


\title[Quasi-redirecting boundaries of non-positively curved groups]{Quasi-redirecting boundaries of non-positively curved groups }

\subjclass[2010]{%
20F65,  
20F67} 

\keywords{quasi-redirecting, relatively hyperbolic groups, Croke-Kleiner admissible groups, 3-manifold groups}

\begin{abstract}
The quasi-redirecting (QR) boundary is a close generalization of the Gromov boundary to all finitely generated groups. In this paper, we establish that the QR boundary exists as a topological space for several well-studied classes of groups. These include fundamental groups of irreducible non-geometric 3-manifolds, groups that are hyperbolic relative to subgroups with well-defined QR boundaries, right-angled Artin groups whose defining graphs are trees, and right-angled Coxeter groups whose defining flag complexes are planar. This result significantly broadens the known existence of QR boundaries.  

Additionally, we give a complete characterization of the QR boundaries of Croke–Kleiner admissible groups that act geometrically on $\mathrm{CAT(0)}$ spaces. We show that these boundaries are non-Hausdorff and can be understood as one-point compactifications of the Morse-like directions.

Finally, we prove that if $ G $ is hyperbolic relative to subgroups with well-defined QR boundaries, then the QR boundary of $ G $ maps surjectively onto its Bowditch boundary. 
 
\end{abstract}  

\maketitle
	\setcounter{tocdepth}{1}
\tableofcontents

\section{Introduction}

Boundaries play a fundamental role in geometric group theory, encoding asymptotic information about groups and their associated spaces. For hyperbolic groups in the sense of Gromov \cite{Gro87}, the Gromov boundary provides a canonical quasi-isometry invariant, with deep connections to dynamics, rigidity, and probability. However, not all non-positively curved groups admit such well-behaved boundaries.  A striking example is provided by Croke–Kleiner groups \cite{CK00}, which act geometrically on CAT(0) spaces yet fail to have well-defined visual boundaries. This highlights the need for alternative boundary constructions that capture meaningful asymptotic properties. Extending this perspective to non-hyperbolic groups has been a central challenge in the field, motivating various boundary constructions that capture ``hyperbolic-like'' directions, including the Morse boundary \cite{CS15}, \cite{Cor17} and the sublinearly Morse boundary \cite{QRT22}, \cite{QRT24}.

Despite these advances, existing boundaries often fail to encode the full range of asymptotic behaviors found in non-hyperbolic settings. The Morse boundary, for instance, is typically non-compact and negligible from the viewpoint of random walks \cite{CDG22}, while sublinearly Morse boundaries, though larger, still omit substantial asymptotic data. To address this, Qing–Rafi \cite{QR24} introduced the quasi-redirecting (QR) boundary, a new quasi-isometry invariant boundary that is often compact and contains sublinearly Morse boundaries as topological subspaces. Unlike previous constructions, the QR boundary captures a richer spectrum of hyperbolic-like behaviors, making it a promising new tool in geometric group theory.
The QR boundary is also shown to serve as a topological model for suitable random walks. When the QR boundary contains at least 3  points, sublinearly Morse boundaries are dense subsets of the QR boundary~\cite{GQV24}. It is also established in~\cite{GQV24} that when $X$ is a rank-one $\CAT(0)$ space, the QR boundary, when it exists, is a visibility space. Moreover, it is shown that when $G$ acts geometrically on a finite-dimensional $\CAT(0)$ cube complex, the QR boundary of $G$, when it exists and contains at least three points,  is not mono-directional and $G$ contains a Morse element. These properties provide evidence that the QR boundary  closely resembles the Gromov boundary of hyperbolic groups.

 The QR boundary is defined as follows:

\begin{defn}
    Let $\alpha, \beta \colon [0, \infty) \to X$ be two quasi-geodesic rays in a metric space $X$. 
     We say $\alpha$ can be \textit{quasi-redirected} to $\beta$ (and write $\alpha \preceq \beta$) if there exists a pair of constants $(q, Q)$ such that for every $r >0$, there exists a $(q, Q)$--quasi-geodesic ray $\gamma$ that is identical to $\alpha$ inside the ball $B(\alpha(0), r)$ and eventually $\gamma$ becomes identical to $\beta$. We say $\alpha \sim \beta$ if $\alpha \preceq \beta$ and $\beta \preceq \alpha$. The resulting set of equivalence classes forms a poset, denoted by $P(X)$. This poset $P(X)$, when equipped with a ``cone-like topology'', is called the \textit{quasi-redirecting boundary} (QR boundary) of $X$ and denoted by $\partial X$.
\end{defn}

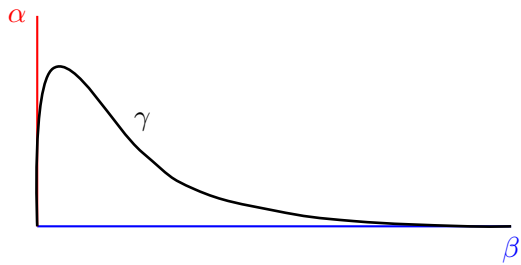
\begin{figure}[h!]
\centering
\begin{tikzpicture}[scale=0.7] 
\draw[blue, thick] (0,0) -- (9,0) node[below] {$\beta$};
\draw[red, thick] (0,0) -- (0,4) node[left] {$\alpha$};
\node at (2,2) {$\gamma$};

 \pgfsetlinewidth{1pt}
 \pgfsetplottension{.75}
 \pgfplothandlercurveto
 \pgfplotstreamstart
 \pgfplotstreampoint{\pgfpoint{0 cm}{0cm}}  
 \pgfplotstreampoint{\pgfpoint{.3 cm}{3 cm}}   
 \pgfplotstreampoint{\pgfpoint{2 cm}{1.4 cm}}
 \pgfplotstreampoint{\pgfpoint{3 cm}{.7 cm}}
 \pgfplotstreampoint{\pgfpoint{4.5 cm}{.3 cm}}
 \pgfplotstreampoint{\pgfpoint{6 cm}{.1 cm}} 
 \pgfplotstreampoint{\pgfpoint{8 cm}{.005 cm}} 
 \pgfplotstreampoint{\pgfpoint{9 cm}{0 cm}} 
 \pgfplotstreamend
 \pgfusepath{stroke} 
 \end{tikzpicture}
\caption{The ray $\alpha$ can be quasi-redirected to $\beta$ at radius $r$.}
\end{figure}

In order to define a ``cone-like topology'' on $P(X)$~\cite{QR24}, $X$ needs to satisfy three assumptions which we call the \emph{QR-assumptions}; see Section~\ref{subsec:QRboundary}. Despite its potential, a major open problem in the theory of QR boundaries is to determine for which groups it is well-defined. It is unknown which groups satisfy all three QR-Assumptions, and consequently, on which groups the QR boundary is defined. In~\cite[Question~D]{QR24}, it is asked whether  all finitely generated groups satisfy the  QR-Assumptions. On the one hand, there are no known examples of  finitely generated groups that do not satisfy the QR-Assumptions. On the other, there are relatively few classes  of  non-hyperbolic  groups that have been verified to satisfy the QR-Assumptions. In this paper, we answer~\cite[Question~D]{QR24} in the affirmative for several classes of groups, significantly extending the known scope of this construction. These class of groups include:
\begin{enumerate}
    \item Croke-Kleiner admissible groups that act geometrically on $\CAT(0)$ spaces.
\item Relatively hyperbolic groups whose peripherals satisfy the QR-assumptions.
\item Fundamental groups of non-geometric 3-manifolds.
\item  Right-angled Coxeter groups whose flag complexes are planar.
\end{enumerate}

Therefore, we provide evidence that the theory of QR-boundaries applies in a variety of concrete contexts. Moreover, our study of 3-manifold groups suggests intriguing connections between the algebraic structure of a group and the height of its QR poset, raising new questions about the relationship between divergence functions and boundary structure. Finally, our results for relatively hyperbolic groups establish a clear link between QR boundaries and Bowditch boundaries, suggesting further potential connections to random walks and the Poisson boundary.

\subsection{3-manifold groups}
Let $M$ be an irreducible non-geometric 3-manifold. The torus decomposition of $M$ yields a nonempty minimal union $\mathcal{T} \subset M$ of disjoint essential tori, unique up to isotopy, such that each component $M_v$ of $M \backslash \mathcal{T}$, called a \emph{piece}, is either  Seifert fibered or hyperbolic.
If all pieces of $M$ are Seifert fibered spaces, then $M$ is called a \emph{graph manifold}. Otherwise, it is called a \emph{mixed manifold}.

We obtain the following result.


\begin{introthm}\label{thm:3-manifold}
    Let $M$ be an irreducible non-geometric 3-manifold. Then $G = \pi_1(M)$ satisfies the QR-Assumptions and hence $\partial G$ is well-defined. Furthermore, 
    \begin{enumerate}
        \item if $M$ is a mixed 3-manifold then $\partial G$ surjects onto the Bowditch boundary of $G$.

        \item If $M$ is a graph manifold then the poset $P(G)$ has the largest element and the other minimal elements are not comparable (see Figure~\ref{poset}). The action of $G$ on $\partial G$ is not minimal.
    \end{enumerate} 
\end{introthm}

\begin{figure}[htb]
\centering 
 \def\svgwidth{0.7\textwidth}
\begingroup%
  \makeatletter%
  \providecommand\color[2][]{%
    \errmessage{(Inkscape) Color is used for the text in Inkscape, but the package 'color.sty' is not loaded}%
    \renewcommand\color[2][]{}%
  }%
  \providecommand\transparent[1]{%
    \errmessage{(Inkscape) Transparency is used (non-zero) for the text in Inkscape, but the package 'transparent.sty' is not loaded}%
    \renewcommand\transparent[1]{}%
  }%
  \providecommand\rotatebox[2]{#2}%
  \newcommand*\fsize{\dimexpr\f@size pt\relax}%
  \newcommand*\lineheight[1]{\fontsize{\fsize}{#1\fsize}\selectfont}%
  \ifx\svgwidth\undefined%
    \setlength{\unitlength}{900.54805424bp}%
    \ifx\svgscale\undefined%
      \relax%
    \else%
      \setlength{\unitlength}{\unitlength * \real{\svgscale}}%
    \fi%
  \else%
    \setlength{\unitlength}{\svgwidth}%
  \fi%
  \global\let\svgwidth\undefined%
  \global\let\svgscale\undefined%
  \makeatother%
  \begin{picture}(1,0.20358709)%
    \lineheight{1}%
    \setlength\tabcolsep{0pt}%
    \put(0,0){\includegraphics[width=\unitlength,page=1]{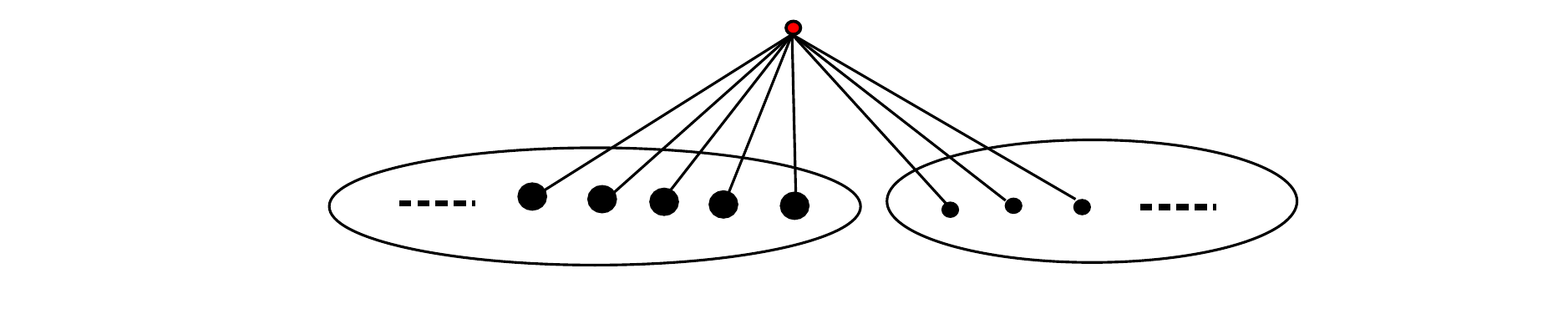}}%
    \put(0.3055771,0.15566084){\color[rgb]{0,0,0}\makebox(0,0)[lt]{\lineheight{1.25}\smash{\begin{tabular}[t]{l}$P(G)$\end{tabular}}}}%
    \put(-0.00049334,0.00507145){\color[rgb]{0,0,0}\makebox(0,0)[lt]{\lineheight{1.25}\smash{\begin{tabular}[t]{l}$\text{sublinear Morse elements}$\end{tabular}}}}%
    \put(0.5576868,0.01675029){\color[rgb]{0,0,0}\makebox(0,0)[lt]{\lineheight{1.25}\smash{\begin{tabular}[t]{l}$\text{not sublinear Morse elements}$\end{tabular}}}}%
    \put(0.52614887,0.18678818){\color[rgb]{0,0,0}\makebox(0,0)[lt]{\lineheight{1.25}\smash{\begin{tabular}[t]{l}$[\zeta^{*}]$\end{tabular}}}}%
  \end{picture}%
\endgroup%

\caption{The picture provides a complete description of the poset $ P(G) $. The largest element, $ [\zeta^*] $, is positioned at the top, while the minimal elements are at the bottom. Both the set of sublinearly Morse elements and the set of non-sublinearly Morse elements have uncountable cardinality. }\label{poset}
\end{figure}
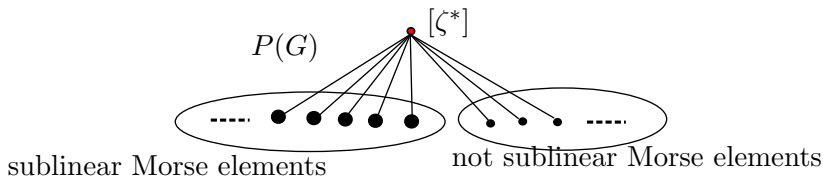

One interesting point here is that Theorem \ref{thm:3-manifold} shows that the height of $ P(G) $ is 2. It is widely known from the results of Kapovich-Leeb \cite{KL98} and Gersten \cite{Ger94} that the fundamental group of graph manifolds has quadratic divergence. So far, for groups with linear divergence, their $ P(G) $ has height 1. It would be interesting to understand the relationship between a group's divergence being a polynomial of degree $ d $ and the height of $ P(G) $.

To prove Theorem~\ref{thm:3-manifold}, we establish the existence of the quasi-redirecting boundary for \textit{Croke-Kleiner admissible groups} and \textit{relatively hyperbolic groups}. These results encompass but extend far beyond the fundamental groups of graph manifolds and mixed manifolds.

\subsection{Croke-Kleiner admissible groups}
When $M$ is an irreducible non-geometric 3-manifold, there is an induced graph of groups decomposition $\mathcal{G}$ of $\pi_1(M)$ with underlying graph $\Gamma$ as follows. For each piece $M_v$, there is a vertex $v$ of $\Gamma$  with vertex group $\pi_1(M_{v})$. For each torus $T_e\in \mathcal T$ contained in the closure of pieces $M_v$ and $M_{w}$, there is an edge $e$ of $\Gamma$ between vertices $v$ and $w$. The associated edge group is $\pi_1(T_e)\cong \Z^2$ and the edge monomorphisms are the maps induced by inclusion. 

 Croke--Kleiner \cite{CK02} defined the class of \textit{admissible groups}, which have a graph of groups decomposition generalizing that of graph manifolds \cite{CK02}. Roughly speaking, a Croke-Kleiner  admissible group $G$ is a graph of groups $(\Gamma, \{G_{\hat v} \}, \{G_{\hat e}\}, \{\tau_{\hat e} \})$ with a nontrivial underlying graph $\Gamma$ where each edge group is $\Z^2$
and each vertex group $G_{\hat v}$ of $G$ has infinite cyclic center $Z_{\hat v}$ with quotient $G_{\hat v}/ Z_{\hat v}$ a non-elementary hyperbolic group. Additionally, the various edge groups need to be pairwise non-commensurable inside each vertex group. For the precise definition of Croke-Kleiner admissible groups, we refer the reader to Definition~\ref{defn:admissible}.


Croke–Kleiner admissible groups are among the simplest non-hyperbolic groups built algebraically from a finite collection of hyperbolic groups. Their study from various perspectives has recently gained attention, and ongoing research continues to explore their properties (see \cite{CarolynNguenRamussen2024}, \cite{HRSS24}, \cite{Tao25}, \cite{SZ24}, \cite{HNY23}, \cite{NY23}, \cite{MN24}, \cite{NQ24}, among others).

In \cite{Wis00}, Wise introduces the concept of an \textit{omnipotent group} which has been widely used in subgroup separability. 

\begin{defn}
\label{defn:omnipotence}
A set of group elements $h_1, \cdots, h_r$ in a group $H$ is called {\it{independent}} if whenever $h_i$ and $h_j$ have conjugate powers then $i =j$.
A group $H$ is {\it{omnipotent}} if whenever $\{h_1,\cdots, h_r\}$ ($r\ge 1$) is an independent set of group elements, then there is a positive integer $p\ge 1$ such that for every choice of positive integers $\{n_1,\cdots, n_r\}$, there is a finite quotient $\varphi \colon H\to \hat H$  such that $\varphi(\hat h_i)$ has order $n_ip$ in $\hat H$  for each $i$. 
\end{defn}

 It is worth mentioning that free groups \cite{Wis00}, surface groups \cite{Baj07}, Fuchsian groups \cite{Wil10}  and virtually special hyperbolic groups \cite{Wis00} all belong to the omnipotent group category. However, it is a longstanding open question whether every hyperbolic group is residually finite. Wise suggested that if every hyperbolic group is residually finite, then any hyperbolic group would be considered an omnipotent group \cite[Remark~3.4]{Wis00}). 


\begin{introthm}\label{thm:introQRofadmissibleCAT0}
Let $G$ be a Croke-Kleiner  admissible group such that each vertex group $G_v$ of $G$ is a CAT(0) group and its quotient $Q_{\hat{v}} = G_{\hat{v}} / Z_{\hat{v}}$ is omnipotent.
Then:
\begin{enumerate}
    \item $G$  satisfies all three QR-Assumptions, ensuring $\partial G$ is defined and is a  quasi-isometry invariant.

    \item The poset $P(G)$ has the largest element while the minimal elements are pairwise incomparable (the picture is simarlar as  Figure~\ref{poset}). The action of $G$ on $\partial G$ is not minimal.

\end{enumerate}
\end{introthm}

{The proof follows from a careful analysis of quasi-geodesic behavior in the group, particularly through the construction of backward spiral paths and forward spiral paths in Section~\ref{sec:template}.

The following corollary is an immediate consequence of Theorem~\ref{thm:introQRofadmissibleCAT0}.

\begin{introcor}
    Let $\Gamma$ be  a finite tree, then the associated right-angled Artin
group $A_\Gamma$ satisfies all three QR-Assumptions and hence $\partial A_\Gamma$ is well-defined.
\end{introcor}

\begin{proof}
    If $\Gamma$ is a simplicial graph, let $A_\Gamma$ denote the associated right-angled Artin group. We study the QR-boundary of $A_\Gamma$ when $\Gamma$ is a finite tree.  If $\Gamma$ consists of a single edge, then $A_\Gamma$ is isomorphic to $\Z^2$, and the QR-boundary of $A_\Gamma$ consists of a single point. If $\Gamma$ contains at least one vertex of degree $\ge 2$ then it is a well-known fact that the associated right-angled Artin group $A_{\Gamma}$ is the fundamental group of a nonpositively curved graph manifold $M$. In particular, $A_\Gamma$ is a Croke-Kleiner admissible group such that each vertex group $G_v$ of $A_\Gamma$ is a direct product of $\Z$ with a free group. As free groups are omnipotent, the conclusion follows from Theorem~\ref{thm:introQRofadmissibleCAT0}. 
\end{proof}

\subsection{Relatively hyperbolic groups}

Relatively hyperbolic groups form a broad and important class of groups that generalize hyperbolic groups. A finitely generated group $G$ is hyperbolic relative to a finite collection of subgroups $\mathbb{P}$  if its coned-off Cayley graph is a hyperbolic and fine graph \cite{Bow12}. This notion captures groups that exhibit hyperbolic behavior outside of specific peripheral subgroups. The Bowditch boundary $\partial_{B} G$  is defined as the boundary of this coned-off graph, providing a compactification that encodes the group's geometric finiteness properties. Establishing the existence of the QR boundary in this setting is crucial for understanding how it interacts with the Bowditch boundary and other geometric structures.

\begin{introthm}\label{thm:introRHGs}
Let $G$ be a group that is hyperbolic relative to  $\mathbb P$. If the quasi-redirecting boundaries exist for each group $P \in \mathbb P$, then the quasi-redirecting boundary of $G$ exists and $\partial G$ surjects onto the Bowditch boundary $\partial_{B} G$ of $(G, \mathbb P)$.

\end{introthm}

In~\cite{QR24}, the authors show that if $(G, \calP)$ is a relatively hyperbolic group where the QR-boundaries of each $P$ is a mono-directional set, i.e.\ $\partial P$ is a point for each $P \in \calP$, then $\partial G$ exists and is homeomorphic to the Bowditch boundary of $(G, \calP)$. In Theorem~\ref{thm:introRHGs}, we drop the assumption that the $P$'s are mono-directional.

\subsection{Right-angled Coxeter groups}
A simplicial complex $\Delta$ is called \emph{flag} if any complete subgraph of the $1$-skeleton of $\Delta$ is the 1-skeleton of a simplex of $\Delta$. Let $\Gamma$ be a finite simplicial graph. The \emph{flag complex} of $\Gamma$ is the flag complex with 1-skeleton $\Gamma$. A simplicial subcomplex $B$ of a simplicial complex $\Delta$ is called \emph{full} if every simplex in $\Delta$ whose vertices all belong to $B$ is itself in $B$.

The flag complex of $\Delta$ is \emph{planar} if it can be embedded into the $2$-dimensional sphere $\field{S}^2$. From now on every time we consider a flag complex it will be as a subspace of the $2$-dimensional sphere $\field{S}^2$. 

\begin{defn}
Given a finite simplicial graph $\Gamma$, the associated \emph{right-angled Coxeter group} $W_\Gamma$ is generated by the set $S$ of vertices of $\Gamma$ and has relations $s^2 = 1$ for all $s$ in $S$ and $st = ts$ whenever $s$ and $t$ are adjacent vertices. The graph $\Gamma$ is the \emph{defining graph} of a right-angled Coxeter group $W_\Gamma$ and its flag complex $\Delta=\Delta(\Gamma)$ is the \emph{defining nerve} of the group. Therefore, sometimes we also denote the right-angled Coxeter group $W_\Gamma$ by $W_\Delta$ where $\Delta$ is the flag complex of $\Gamma$. 

Let $S_1$ be a subset of $S$. The subgroup of $W_\Gamma$ generated by $S_1$ is a right-angled Coxeter group $W_{\Gamma_1}$, where $\Gamma_1$ is the induced subgraph of $\Gamma$ with vertex set $S_1$ (i.e.\ $\Gamma_1$ is the union of all edges of $\Gamma$ with both endpoints in $S_1$). The subgroup $W_{\Gamma_1}$ is called a \emph{special subgroup} of $W_\Gamma$.
\end{defn}

\begin{introcor}[Theorem~\ref{maincox}]
\label{cor:RACG}
  Let $\Gamma$ be a graph whose flag complex $\Delta$ is planar.   Then the right-angled Coxeter group $W_\Gamma$ satisfies all three QR-Assumptions.
\end{introcor}

\subsection{Open questions and future directions}
The quasi-redirecting boundary (QR boundary) introduced in \cite{QR24} provides a new perspective on boundaries of groups. We have established its well-definedness for several important classes of groups.

However, many questions remain open. Below, we outline several directions for future research.

\begin{enumerate}
    \item Does QR boundary exist for weakly hyperbolic groups, acylindrically hyperbolic groups, or hierarchically hyperbolic groups? Alternatively, we ask does QR boundary exist for all finitely generated groups?

    \item What is the QR boundary of CB-generated but not CB groups such as big mapping class groups of suitable surfaces?

    \item In what cases is the Martin boundary a subset of the QR boundary? 

    \item How does the QR boundary of a free-by-cyclic group reflect the algebraic or dynamical properties of the monodromy? For instance, when the monodromy is hyperbolic, can QR boundaries provide new insights into Cannon–Thurston maps?
\end{enumerate}

We hope that the techniques and results in this paper will inspire further developments in the study of QR boundaries and their applications in geometric group theory.

\subsection*{Overview} 
In Section~\ref{sec:background}, we review the preliminaries on quasi-redirecting boundaries and the necessary background on quasi-geodesics. Section~\ref{sec:QR-boundary-RHGs} establishes the existence of QR boundaries for relatively hyperbolic groups and proves their surjectivity onto the Bowditch boundary (Theorem~\ref{thm:introRHGs}). Section~\ref{sec:template} explores quasi-geodesics in templates, a key step in understanding QR boundaries of Croke-Kleiner admissible groups, which we address in Section~\ref{sec:QR-admissible-groups}. Finally, the proofs of Theorem~\ref{thm:3-manifold} and Corollary~\ref{cor:RACG} are presented in Sections~\ref{sec:3-manifolds} and \ref{sec:QR-RACG}, respectively.

\ref{sec:QR-RACG}, respectively.
 
\subsection*{Acknowledgments} 
The authors are very grateful to Alex Margolis for several critical readings of earlier drafts and especially for his help in proving Proposition~\ref{prop:smallslope}.

\section{Preliminaries}
\label{sec:background}
In this section, we recall the construction of quasi-redirecting boundary as presented in~\cite{QR24}. We refer  to~\cite{QR24} for a complete treatment.

Let $X$ and $Y$ be metric spaces and $f$ be a map from $X$ to $Y$. Let ${\qq}=(q, Q) \in[1,\infty) \times [0,\infty)$ be a pair of constants. 
\begin{enumerate}
	
	\item We say that $f$ is a \emph{$\qq$--quasi-isometric embedding} if  for all $x, y\in X$,
	      \[
		        \frac{1}{q} d(x, x') - Q \le d(f(x), f(x')) \le q d(x,x') + Q.
	      \]

	\item We say that $f$ is a \emph{$(q,Q)$--quasi-isometry} if it is a $\qq$--quasi-isometric embedding such that $Y = N_{Q}(f(X))$.

\end{enumerate}

\subsection{Quasi-redirecting boundary}\label{subsec:QRboundary}
  Let $X$ be a proper geodesic metric space.
\begin{defn}[Quasi-Geodesics]\label{Def:Quadi-Geodesic} 
A \emph{quasi-geodesic} in a metric space $X$ is a Lipschitz quasi-isometric embedding 
$\alpha: I \to X$ where $I \subset \mathbb{R}$ is a 
(possibly infinite) interval.  We use $\qq=(q, Q)$ to indicate the constants, so that  $\alpha :  I \to X$ is a $\qq$--quasi-geodesic if for all $s, t \in I$, we have 
\[
\frac{|t-s|}{q} - Q  \leq d_X \big(\alpha(s), \alpha(t)\big)  \leq q |s-t|. 
\]
\end{defn}

The assumption that $\alpha$ is Lipschitz is needed so we can apply the Arzel\`a--Ascoli theorem to a sequence of quasi-geodesics 
and obtain a limiting quasi-geodesic. However, the assumption that a quasi-isometric embedding $\alpha:I\to X$ is Lipschitz  can be achieved after replacing $\alpha$ with a quasi-geodesic fellow-traveling $\alpha$ (\cite[Lemma 2.3]{QR24}).

\subsection{Notation} Let $\go$ be a fixed base-point in  $X$. We use ${\qq}=(q, Q) \in[1, \infty) \times[0, \infty)$ to indicate a pair of constants. 
 For instance, one can say $\Phi \colon X \to Y$ is a ${\qq}$--quasi-isometry and $\alpha$ is a ${\qq}$--quasi-geodesic ray or segment. 

By a ${\qq}$--ray we mean a ${\qq}$--quasi-geodesic ray $\alpha:[0, \infty) \to X$ such that $\alpha(0)= \gothic o$. For an interval $[s, t] \subset[0, \infty)$, we denote the restriction of $\alpha$ to the time interval $[s, t]$ by $\alpha[s, t]$. 
However, if points $x, y \in X$ on the image of $\alpha$ are given, we denote the sub-segment of $\alpha$ connecting $x$ to $y$ by $[x, y]_\alpha$. That is, if $\alpha(s)=x$ and $\alpha(t)=y$ for $s \leq t$, then $[x, y]_\alpha=\alpha[s, t]$.

Let $\alpha:\left[s_1, s_2\right] \rightarrow X$ and $\beta:\left[t_1, t_2\right] \rightarrow X$ be two quasi-geodesics such that $\alpha\left(s_2\right)=\beta\left(t_1\right)$. In this paper we denote the concatenation of $\alpha$ and $\beta$ by $\alpha \cup \beta$ by which we mean the following:
\[
\alpha \cup \beta:\left[s_1, t'\right] \rightarrow X, \quad (\alpha \cup \beta)(t)=
\begin{cases} \alpha(t) & \textrm{for } t \in\left[s_1, s_2\right] \\
    \beta\left(t+t_1-s_2\right) & \textrm {for } t \in[s_2, t']\end{cases}
\]
where $t'\coloneqq t_2-t_1+s_2$.

For $r>0$, let $B_r^{\circ} \subset X$ be the open ball of radius $r$ centered at $\mathfrak{o}$, let $B_r$ be the closed ball centered at $\mathfrak{o}$ and let $B_r^c=X-B_r^{\circ}$.

For a $\qq$--ray $\alpha$ and $r>0$, we let $t_r \geq 0$ denote the first time when $\alpha$ first intersects $B_r^c$ and $T_r \geq t_r$ be the last time $\alpha$ intersects $B_r$. We denote $\alpha\left(t_r\right)$ by $\alpha_r \in X$. Also, let
\[
\left.\alpha\right|_r \coloneqq\alpha\left[0, t_r\right] \quad  
\text { and }\quad \alpha|_{\geq r} \coloneqq \alpha [T_r, \infty) 
\] 
be the restrictions $\alpha$ to the intervals $\left[0, t_r\right]$ and $\left[T_r, \infty\right)$ respectively. 
That is, $\left.\alpha\right|_r$ is the subsegment of $\alpha$ connecting $\gothic{o}$ to $\alpha_r$, and $\left.\alpha\right|_{\geq r}$ is the portion of $\alpha$ that starts at radius $r$ and never returns to $B_r$.

Lastly, if $p$ is a point on a ${\qq}$--ray $\alpha$,  we use $\alpha_{[p, \infty)}$ to denote the tail of $\alpha$ starting from the point $p$. 
Note such a point always exists as a quasi-geodesic is always assumed to be a ray without loss of generality. This is because, as discussed in~\cite[Definition 2.2]{QRT22}, one can adjust the quasi-isometric embedding of an interval slightly to make it continuous (see~\cite[Lemma III.1.11]{BH99}). 

We also use $d(\cdot, \cdot)$ instead of $d_X(\cdot, \cdot)$ when the metric space $X$ is fixed. For $x \in X,\|x\|$ denotes $d(\gothic{o}, x)$.
Now we recall the first of the three QR-Assumptions.

\addtocontents{toc}{\protect\setcounter{tocdepth}{0}}
\subsection*{QR-Assumption 0} (No dead ends) 
The metric space $X$ is proper and geodesic. 
Furthermore, there exists a pair of constants ${\qq}_0$ such that every point $x \in X$ lies on an infinite ${\qq}_0$--quasi-geodesic ray.  

\begin{rem}
QR-Assumption 0 is satisfied by the Cayley graph of an infinite finitely generated group with  respect to a finite generating set~\cite[Lemma~2.5]{QR24}.
\end{rem}

\begin{defn}\label{Def:Redirection}
Let $X$ be a geodesic metric space.
Let $\alpha, \beta$ and $\gamma$ be quasi-geodesic rays in $X$. We say
\begin{enumerate}
    \item $\gamma$ \emph{eventually coincides with 
$\beta$} if there are times $t_\beta, t_\gamma >0$ such that, 
for $t\geq t_\gamma$, we have 
$
\gamma(t) =\beta(t+t_\beta)
$.

\item For $r>0$, we say $\gamma$ \emph{quasi-redirects $\alpha$ to $\beta$ at radius $r$} if
$
\gamma|_r = \alpha|_r$  and $\beta$ 
eventually coincides with $\gamma$.
 If $\gamma$ is a ${\qq}$--ray, we say \emph{$\alpha$ can be ${\qq}$--quasi-redirected to $\beta$
at radius $r$} or \textit{$\alpha$ can be ${\qq}$--quasi-redirected to $\beta$
by $\gamma$ at radius $r$}. We refer to $t_\gamma$ as the \emph{landing time}. 

\item We say $\alpha$ is \textit{quasi-redirected} to $\beta$, denoted by $\alpha \preceq \beta$, if there is ${\qq} \in [1, \infty) \times[0,\infty)$ such that for every $r>0$, $\alpha$ can be ${\qq}$--quasi-redirected to $\beta$ at radius $r$. 
\end{enumerate}
\end{defn}

\begin{figure}[ht]
\centering
\begin{tikzpicture}[scale=0.8] 
 \tikzstyle{vertex} =[circle,draw,fill=black,thick, inner sep=0pt,minimum size=.5 mm]
  \node[vertex] (go) at (0,0)[label=left:$\go$] {}; 
	\draw[thick,dashed] (0,0) -- (5,0) ; 
	\draw[thick,dashed] (0,0) -- (0,4) node[above] {$\alpha$}; 
	\draw [blue,ultra thick] (0,0) -- (0,1) -- (3,1) node[above] {$\gamma$} -- (3,0) -- (4.5,0) ; 
	\draw (1.5,0) node[below] {$\beta$} ; 
\end{tikzpicture}
\begin{tikzpicture} [scale=0.8]
 \tikzstyle{vertex} =[circle,draw,fill=black,thick, inner sep=0pt,minimum size=.5 mm]
  \node[vertex] (go) at (0,0)[label=left:$\go$] {}; 
		\draw[thick,dashed] (0,0) -- (5,0); 
		\draw[thick,dashed] (0,0) -- (0,4) node[above] {$\alpha$}; 
		\draw [blue,ultra thick] (0,0) -- (0,2.5) -- (2,2.5) node[above] {$\gamma$}  -- (2,0) -- (4.5,0) ; 
		\draw (1,0) node[below] {$\beta$} ;
\end{tikzpicture}
\caption{Two instances in which $\alpha$ can be quasi-redirected to $\beta$ by $\gamma$. Here, $\alpha$ and $\beta$ are shown as dashed lines, and $\gamma$ is shown as a solid blue line.}
\end{figure}
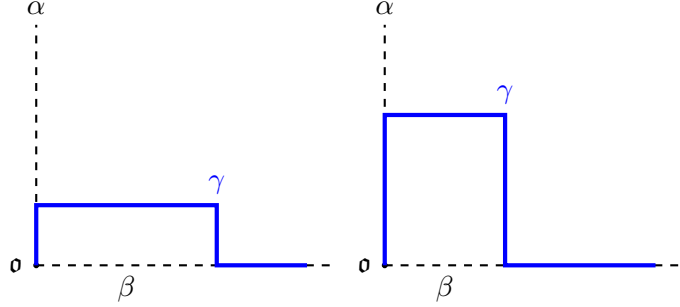

\begin{defn}
Define $\alpha \simeq \beta$ if and only if $\alpha  \preceq \beta$ and $\beta  \preceq \alpha$. Then  
$\simeq$ is an equivalence relation on the space of all quasi-geodesic rays in $X$. Let $P(X)$
denote the set of all equivalence classes of quasi-geodesic rays under $\simeq$. 
For a quasi-geodesic ray $\alpha$, let $[\alpha] \in P(X)$ denote the equivalence class containing 
$\alpha$. We extend $\preceq$ to $P(X)$ by defining $[\alpha] \preceq [\beta]$ if 
$\alpha \preceq \beta$. Note that this does not depend on the chosen representative
in the given class. The relation $\preceq$ is a partial order on elements of $P(X)$.
\end{defn}

\begin{lem}[{\cite[Lemma 3.2]{QR24}}]\label{lem:transitive}
    Let $\alpha, \beta, \gamma$ be quasi-geodesic rays. Suppose that $\alpha$ can be $(q_1, Q_1)$--quasi-redirected to $\beta$ at radius $r$, and that $\beta$ can be $(q_2, Q_2)$--quasi-redirected to $\gamma$ at every radius. Then $\alpha$ can be $(q_3, Q_3)$--quasi-redirected to $\gamma$ at radius $r$, where $q_3 = \max \{q_1, q_2 +1 \}$ and $Q_3 = \max \{Q_1, Q_2 \}$.
\end{lem}

\subsection*{QR-Assumption 1} (Quasi-geodesic representative) 
For ${\qq}_0$ as in QR-Assumption 0, every equivalence class of quasi-geodesics $\mathbf{a} \in P(X)$ contains a ${\qq}_0$--ray. 
We fix such a ${\qq}_0$--ray, denote it by $\underline{a} \in \mathbf{a}$, and call it a \emph{central element} of $\bfa$.

\subsection*{QR-Assumption 2} (Uniform redirecting function) 
\addtocontents{toc}{\protect\setcounter{tocdepth}{1}}
For every $\mathbf{a} \in P(X)$, there is a function 
\[
f_\mathbf{a} : \, [1, \infty) \times [0, \infty) \to [1, \infty) \times [0, \infty), 
\] 
called the redirecting function of
the class $\mathbf{a}$, such that if $\mathbf{b} \prec \mathbf{a}$ then any ${\qq}$--ray $\beta \in \mathbf{b}$ can be
$f_\mathbf{a}({\qq})$--quasi-redirected to $\underline{a}$.

\begin{prop}[{\cite[Proposition~4.3]{QR24}}]\label{prop:product}
 Let $X = A \times B$   where $A$ and $B$ are proper metric spaces satisfying QR-Assumption~0, equipped with $L^{\infty}$--metric. Then $P(X)$ is a point.
\end{prop}
Note that since $P(X)$ is invariant under quasi-isometries, Proposition~\ref{prop:product} also holds if we
equip $X$ with the $L^p$--metric with $p >0$.

\subsection{Topology on \texorpdfstring{$X \cup P(X)$}{X ∪ P(X)}}
The topology on $X \cup P(X)$ is defined as follows. Recall that points
in $P(X)$ are equivalence classes of quasi-geodesic rays. For each point $x\in X$, we define
\[
 \mathbb{x} = \Big\{ \text{quasi-geodesics rays passing through $x$} \Big\}. 
\]
We let $\Ga, \Gb, \Gc$  denote elements of $P(X) \cup X$, that is, either a set of quasi-geodesic rays 
passing through a point $x \in X$ or an equivalence class of quasi-geodesic rays in $P(X)$. For $\bfa \in P(X)$, define
$F_\bfa : \, [1,\infty) \times [0, \infty) \to [1,\infty) \times [0, \infty)$ by 
\begin{equation}\label{actualtopology}
F_\bfa(\qq) = \max \{ {\mathbf f}_\bfa(\qq) + (1,0), (4q+3Q)\} \qquad\text{for}\qquad \qq \in   [1,\infty) \times [0, \infty). 
\end{equation}
\begin{defn}\label{Def:The-In-Topology}
For $\bfa \in P(X)$ and $r>0$, define   
\begin{align*}
 \calU(\bfa, r) \coloneqq \Big  \{  &\Gb \in P(X) \cup X \,  \Big\vert\,
 \text{every $\qq$--ray in $\Gb$ can be $F_\bfa(\qq)$--quasi-redirected to $\ua$ at radius $r$} \Big \}.  \end{align*}
\end{defn}

\subsection*{A system of neighborhoods}
For each $\bfa \in P(X)$, recall that  
\[
\calB(\bfa) = \Big\{ \calV \subset X \cup P(X) \,\Big\vert\,
\calU(\bfa, r) \subset \calV   \text{ for some $r>0$} \Big\}
\]
and for every $x \in X$, define 
\[
\calB(\mathbb{x}) = \Big\{ \calV \subset X \cup P(X) \,\Big\vert\,
B(\mathbb{x}, r) \subset \calV  \text{ for some $r>0$} \Big\}.
\]
We thus define the topology on $X \cup P(X)$ so that $\calB(\bfa)$ and $\calB(\mathbb{x})$ are a system of neighborhoods for each $\bfa \in P(X)$ and $x \in X$ respectively.
We collect some important facts from~\cite{QR24} about the QR boundary and the poset $P(X)$.
\begin{theorem}[{\cite[Theorem B]{QR24}}]\label{thm:collectfacts}
Let $X, Y$ be proper geodesic metric spaces satisfying all three QR-Assumptions.
\begin{enumerate}
\item Suppose that $\Phi: X \rightarrow Y$ is a quasi-isometry sending the base point $\mathfrak{o}_X \in X$ to the base point $\mathfrak{o}_Y \in Y$. Then there is a well-defined induced map
\[
\Phi^*: P(X) \rightarrow P(Y) \quad \text { where } \quad \Phi^*([\alpha])=[\Phi \circ \alpha]. \]
Furthermore, $\Phi^*$ preserves the partial order on $P(X)$ and $P(Y)$.
\item $\partial X$ and $X \cup \partial X$ are QI-invariant as topological spaces.
\item Sublinearly Morse boundaries are topological subspaces of $\partial X$.
\end{enumerate}
\end{theorem} 

\subsection{Surgery on quasi-geodesics}
We recall a few surgeries related to quasi-geodesics that will often be used in the subsequent arguments.
\begin{lem}[{\cite[Lemma 2.6]{QR24}}]\label{lem:surgery1}
Let $X$ be a metric space that satisfies QR-Assumption 0. 
\begin{enumerate}[label={(\arabic*)},ref={\thecor(\arabic*)}]
    \item\label{concate} (Nearest-point projection surgery)
Consider a point $x \in X$ and a $(q, Q)$--quasi-geodesic segment $\beta$ 
connecting a point $z \in X$ to a point $w \in X$. Let $y$ be a closest point in $\beta$
to $x$. Then 
\[
\gamma = [x, y] \cup [y, z]_\beta
\] 
is a $(3q, Q)$--quasi-geodesic.

\begin{figure}[!ht]
\begin{tikzpicture}[scale=0.9]
 \tikzstyle{vertex} =[circle,draw,fill=black,thick, inner sep=0pt,minimum size=.5 mm]

  \node[vertex] (z) at (0,0)[label=left:$z$] {}; 
  \node[vertex] (w) at (8,0) [label=right:$w$]  {}; 
  \node[vertex] (x) at (4,2) [label=right:$x$]  {};     
  \node[vertex] (y) at (4,0) [label=below:$y$]  {};     
  \node (a) at (.9,.8) [label=right:$\beta$]  {};    
  \draw[thick, dashed]  (4, 2)--(4, 0){};
 
  \pgfsetlinewidth{1pt}
  \pgfsetplottension{.75}
  \pgfplothandlercurveto
  \pgfplotstreamstart
  \pgfplotstreampoint{\pgfpoint{0cm}{0cm}}  
  \pgfplotstreampoint{\pgfpoint{1.4cm}{-.6cm}}   
  \pgfplotstreampoint{\pgfpoint{1.3cm}{.2cm}}
  \pgfplotstreampoint{\pgfpoint{3cm}{-.4cm}}
  \pgfplotstreampoint{\pgfpoint{4cm}{0cm}}
  \pgfplotstreampoint{\pgfpoint{5cm}{-.2cm}}
  \pgfplotstreampoint{\pgfpoint{6cm}{.3cm}}
  \pgfplotstreampoint{\pgfpoint{7cm}{-.7cm}}
  \pgfplotstreampoint{\pgfpoint{8cm}{0cm}}
  \pgfplotstreamend
  \pgfusepath{stroke} 
  \end{tikzpicture}
  
\caption{The concatenation of the geodesic segment $[x,y]$ 
and the quasi-geodesic segment $[y,z]_\beta$ is a quasi-geodesic.}\label{Fig:omega} 
\end{figure}
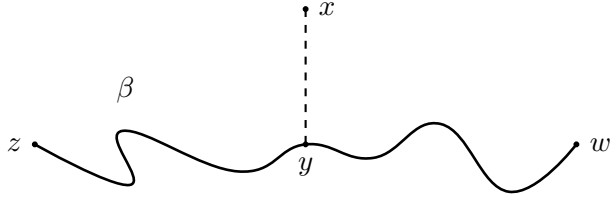

\item\label{redirect11} (Quasi-geodesic ray to geodesic ray surgery) Let $\beta$ be a geodesic ray and 
$\gamma$ be a $(q, Q)$--ray. 
For $r>0$, assume that $d_X(\beta_r, \gamma)\leq r /2$. Then there exists a
$(9q,Q)$--quasi-geodesic $\gamma'$ where $\gamma'(t) =\beta(t)$ for large values of $t$ and 
\[
\gamma|_{r/2} = \gamma'|_{r/2}. 
\]
\item (Segment to quasi-geodesic ray surgery)\label{item_surgery_segment}
Consider a $(q, Q)$-quasi-geodesic ray  $\alpha \colon [0, \infty) \to X$ 
and a finite 
$(q, Q)$--quasi-geodesic segment $\beta \colon [a,b] \to X$. Then there is 
$s_0 \in [0, \infty)$ such that the following holds: for $ s \in [s_0, \infty)$ let $s_\gamma \in [s, \infty)$ and $t_\gamma \in [a,b]$ be such that  
$[\beta(t_\gamma), \alpha(s_\gamma)]$ is a geodesic segment that realizes the set distance between 
$\alpha[s, \infty)$ and $\beta$. 
 Then  
\[
\gamma = \beta[a, t_\gamma] \cup [\beta(t_\gamma), \alpha(s_\gamma)] \cup \alpha[s_\gamma, \infty)  
\] 
is a $(4q, 3Q)$--quasi-geodesic.

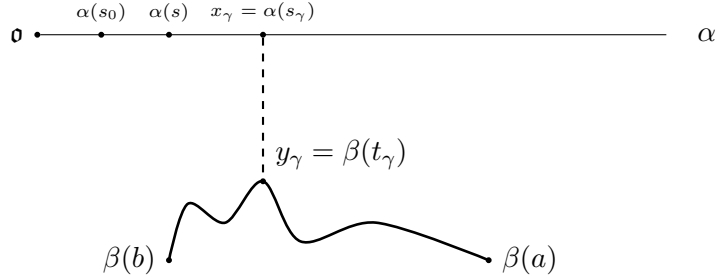
\begin{figure}[h!]
\begin{tikzpicture}[scale=0.5]
 \tikzstyle{vertex} =[circle,draw,fill=black,thick, inner sep=0pt,minimum size=.5 mm]
                  

  \node[vertex] (o) at (-3,5)[label=left:$\go$] {}; 
  \node[vertex]  at (-1.3,5)[label=above:\tiny $\alpha(s_0)$] {}; 
  \node[vertex]  at (.5,5)[label=above:\tiny $\alpha(s)$] {}; 
  \node(a) at (14,5)[label=right:$\alpha$] {}; 
  \draw (o)--(a){};

  \node[vertex](c) at(3, 5)[label=above:\tiny ${x_\gamma=\alpha(s_\gamma)}$]{}; 
  
  \node[vertex](d) at(3, 1.1)[label=above right:${y_\gamma=\beta(t_\gamma)}$] {}; 
  \node[vertex] at (.5, -1) [label=left:$\beta(b)$] {}; 
  \node[vertex] at (9, -1) [label=right:$\beta(a)$] {}; 

  \draw[thick, dashed]  (c) to (d){};
  
  \pgfsetlinewidth{1pt}
  \pgfsetplottension{.55}
  \pgfplothandlercurveto
  \pgfplotstreamstart
  \pgfplotstreampoint{\pgfpoint{0.5cm}{-1cm}}
  \pgfplotstreampoint{\pgfpoint{1cm}{.5cm}}
  \pgfplotstreampoint{\pgfpoint{2cm}{0cm}}
  \pgfplotstreampoint{\pgfpoint{3cm}{1.1cm}}
  \pgfplotstreampoint{\pgfpoint{4cm}{-.5cm}}
  \pgfplotstreampoint{\pgfpoint{6cm}{-0cm}}
  \pgfplotstreampoint{\pgfpoint{9cm}{-1cm}}

  \pgfplotstreamend
  \pgfusepath{stroke} 
\end{tikzpicture}
\caption{Segment-to-geodesic-ray surgery.}
\end{figure}

\item (Fellow-traveling surgery)\label{item_fellowtravel}  Let ${\qq}$-rays $\alpha, \beta$ and $t_0>0$ be such that, for all $t \leq t_0$, we have
$
d(\alpha(t), \beta(t)) \leq 1 
$.
Then there exists a $(q, Q+1)$-quasi-geodesic ray $\beta^{\prime}$ such that
\[
\left.\beta^{\prime}\right|_{t_0}=\beta|_{t_0} \quad \text { and }\quad \beta^{\prime}|_{\left(t_0+1, \infty\right)}=\left.\alpha\right|_{\left(t_0, \infty\right)}.
\]
\end{enumerate}
 \end{lem}

\begin{lem}\label{lem:infinitepoints}

	Let $\alpha,\beta$ be quasi-geodesic rays. Suppose there exists constants  $\qq$ and  a sequence of points $\{x_n\}$ on $\alpha$ such that $\norm{x_n} \to \infty$ and the following holds. For every 
  $n$, there exists a $\qq$-ray $\gamma_n$ such that
  $\gamma_n$ eventually coincides with $\beta$, and $\gamma_n$ and $\alpha$ are identical on the subsegment $[\gothic{o}, x_{n}]_{\alpha}$. Then  $\alpha$ can be $\qq$-quasi-redirected to $\beta$.
\end{lem}

\begin{proof}
    Let $s_n$ be the first time in $[0, \infty)$ such that $\alpha(s_n) = x_n$. 
    Consider the ball $B_{r_n}$, where $r_n \coloneqq \norm{x_n}$. Let $t_n$ be the first time $\alpha$ intersects $B_{r_n}^c$. It follows that $t_n \leq s_n$. According to the assumption, $(\gamma_n)|_{r_n} = \alpha|_{r_n}$ and $\gamma_n$ eventually coincides with $\beta$. Fix $r>0$. As $r_n \to \infty$, there exists $n$ such that $r_n \geq r$, and so $\alpha|_r=\gamma_n|_r$. This guarantees that $\alpha$ is $\qq$-quasi-redirected to $\beta$ at radius $r$ via $\gamma_n$. Consequently, $\alpha \preceq \beta$.
    \end{proof}

\section{QR boundary of relatively hyperbolic groups}
\label{sec:QR-boundary-RHGs}
In this section, we analyze the case when $X$ is a Cayley graph of a finitely generated relatively hyperbolic group pair $(G, \calP)$, where $G$ is a group and $\calP$ is a collection of infinite finitely generated subgroups. 
In~\cite{QR24}, the authors show that if $(G, \calP)$ is a relatively hyperbolic group where the QR-boundaries of each $P$ is a mono-directional set, i.e.\ $\partial P$ is a point for each $P \in \calP$, then $\partial G$ exists and is homeomorphic to the Bowditch boundary of $(G, \calP)$. In this section, we drop the assumption that the $P$'s are mono-directional. We  show that if  $\partial P$ exists for all $P \in \calP$, the quasi-redirecting boundary of $(G, \calP)$ exists. Furthermore, we show in Theorem~\ref{surjective} that when it exists, $\partial G$ maps surjectively onto the Bowditch boundary of  $(G, \calP)$.

\subsection{Redirecting in relatively hyperbolic groups}
We present definitions and relevant facts regarding the coarse geometry of relatively hyperbolic groups, which can be found in~\cite{QR24, DS05, Hru10} and~\cite{Sis12}.

\begin{defn}
Fix a finite generating set $S$ once and for all, and let $\Cay(G)$ denote the Cayley graph of $G$ with respect to this 
generating set. We refer to the subgroups $P \in \calP$ as \emph{peripheral} subgroups. Let $\calA$
be the set of subgraphs of $\Cay(G)$ associated to cosets of subgroups in $\calP$. Namely, 
for $P \in \calP$ and $g \in G$, $A_{P, g}$ is the induced subgraph of $\Cay(G)$ with vertex set $gP$. 
We form the \emph{coned-off} Cayley graph, denoted $K(G)$ or simply $K$,  by adding a vertex $*p_A$ for each 
$A \in \calA$, and adding edges of length $\frac12$ from $*p_A$ to each vertex of $A$. 
Since $\Cay(G)$ is a subgraph of $K$, for any two vertices $v, w \in \Cay(G)$, we have 
\begin{equation}\label{shorter}
d_{K} (v, w) \leq d_{\Cay(G)}(v,w).
\end{equation}
\end{defn}

\begin{defn} 
 A graph is \emph{fine} if for each integer $n$, every edge belongs to only finitely many simple cycles of length $n$. If the coned-off Cayley graph $K$ is both hyperbolic and fine, then $G$ is \emph{hyperbolic relative to $\calP$}.

  \end{defn}
 \begin{defn}[Bounded Coset Penetration]
 A key property of a relatively hyperbolic group is \emph{Bounded Coset Penetration}~\cite{Far98}, which we now state. An oriented path $\ell \in K$ is said to \emph{penetrate} $A \in \calA$ if it passes through the cone point $*p_A$ of $A$; 
its \emph{entering} and \emph{exiting} vertices are the vertices immediately before and after $*p_A$ on $\ell$. The path is \emph{without backtracking}
 if once it penetrates $A \in \calA$, it does not penetrate $A$ again.  If for each $q\geq 1$, there is a constant $a = a(q)$ such that if $\zeta, \zeta' \subset K$ are $(q, 0)$--quasi-geodesics without backtracking in $K$ and with the same pair of endpoints, then
 \begin{enumerate}
\item if $\zeta$ penetrates some $A \in \calA$, but $\zeta'$ does not, then the distance between the entering and exiting vertices of $\zeta$ in $A$ is at most $a(q)$; and

\item if $\zeta$ and $\zeta'$ both penetrate $A \in \calA$, then the distance between the entering vertices of $\zeta$ and $\zeta'$ in $A$ is at most $a(q)$, and similarly for the exiting vertices.
\end{enumerate}
\end{defn}
We note that if $(G,\cP)$ is relatively hyperbolic, then there are only finitely many  peripheral subgroups in $\cP$.

For the rest of this section, let $X = \Cay(G)$ denote the Cayley graph of $(G, \calP)$. 
\begin{defn}\cite[Definition 3.9]{Sis12}\label{deep}
 Let $\alpha$ be a path in $X$. For $M, c>0$, define  $\deep_{M,c} (\alpha)$ to be the set of points
 $x \in \alpha$ such that there exists a subpath of $\alpha$ containing $x$ 
 with endpoints $x_1,x_2$ and $A \in \calA$ where
 \[
x_1,x_2 \in N_M(A)  
\qquad\text{and}\qquad 
d(x,x_i) \geq c \quad \text{for $i = 1, 2$}. 
\]
Thinking of $\alpha$ as a subset of $X$, define 
\[
\trans_{M,c}(\alpha) = \alpha \setminus \deep_{M,c} (\alpha)
\]
to be the set of $(M,c)$--transition points of $\alpha$. 
\end{defn}

\begin{proposition}[{\cite{Sis12, DS05}}]\label{sis12}
Let $X = \Cay(G)$.
For every $\qq$ there exist constants $M = M(\qq)$, $c=c(\qq)$,  $D=D(\qq)$ and $\rho(\qq)$ such that the followings hold. 
Let $\alpha: \, [a,b] \to X$ be a $\qq$--quasi-geodesic segment.
\begin{enumerate}
\item  The set $\deep_{M, c}(\alpha)$ is a disjoint union 
of subpaths, each of which is  contained in $N_{\rho M}(A)$ for distinct sets $A \in \calA$. 
\item For any pair of $\qq$--quasi-geodesic 
segments $\alpha, \beta$ with the same endpoints, we have
\[ 
d_{\rm Haus} \big( \trans_{M,c}(\alpha), \trans_{M,c}(\beta)\big) \leq D.
\]

\item Moreover, for every $A \in \calA$ there are times $s, t \in [a,b]$ such that:
\begin{itemize}
	\item  During the interval $[a, s]$, $\alpha$ approaches $A$ at a linear speed.
	\item  During the interval $[t, b]$, $\alpha$ moves away from $A$ at a linear speed.
	\item  $\alpha[s,t] \subset N_{\rho M} (A)$. 
\end{itemize}
\end{enumerate}
The same also holds for quasi-geodesic rays. 
\end{proposition}

The statements of (1) and (2) are contained~\cite[Proposition 5.7]{Sis12}. The statement (3) follows from~\cite[Lemma 4.17]{DS05}.

 \begin{defn}\label{Def:saturation}
Let $\alpha$ be a $\qq$--ray or $\qq$--segment in $X$. The \emph{saturation} of $\alpha$, denoted by $\Sat(\alpha)$, is 
the union of $\alpha$ and all $A \in \calA$ with $N_{M(\qq)}(A) \cap \alpha \neq \emptyset$, where $M(\qq)$ is as in Proposition~\ref{sis12}. 
\end{defn} 

Saturations are quasi-convex (see~\cite[Lemma 4.25]{DS05}):

\begin{lemma}[Uniform quasi-convexity of saturations]\label{Lem:sat-convex}
 For every $\qq$, there exists $\tau(\qq)>0$ such that for every $L>1$ and every $\qq$--ray or $\qq$--segment $\alpha$, $\Sat(\alpha)$ has the property that, for every $\qq$--segment $\gamma$ with endpoints in $N_L(\Sat(\alpha))$, 
 we have  
 \[
  \gamma \subset N_{\tau(\qq) \cdot L}(\Sat(\alpha)).
 \] 
\end{lemma}

Quasi-convexity of saturations allows us to understand quasi-geodesic rays by considering the parabolic sets near which they pass. The subsequent definitions and results make this concrete.

 \begin{defn}\label{Def:transition} 
 Let $\alpha$ be a $\qq$--quasi-geodesic segment or $\qq$--ray in $X$. We say a point $\alpha(t)$ 
 is a \emph{$\qq$--transition point} of $\alpha$ if
\[
\alpha(t)  \in \trans_{M(\qq), c(\qq)}(\alpha),
\] 
where $M(\qq), c(\qq)$ are as in Proposition~\ref{sis12}. 
\end{defn}
 
\begin{defn}\label{Def:Transient} 
Let $\alpha$ be a $\qq$--ray. We say $\alpha$ is a \emph{$\qq$--transient} ray if,
there is a sequence of times $t_i \to \infty$ such that $\alpha(t_i)$ is a 
$\qq$--transition point of $\alpha$. 
\end{defn}

Note that if $\qq' \geq \qq$ and $\alpha$ is a $\qq$--ray, then $\alpha$ is also a 
$\qq'$--ray. But, the set of $\qq$--transition points is not necessarily a subset or a superset
of the set of $\qq'$--transition points because to ensure 
\[
 \deep_{M_1,c_1} (\alpha)  \subseteq \deep_{M_2,c_2} (\alpha),
\] 
we require $c_1 \geq c_2$ and  $M_1\leq M_2$. However, as we shall see, the property of being
a transient ray is independent of the choice of $\qq$. We summarize here that there are exactly two disjoint scenarios of redirecting based on whether a ray is transient or not.
\begin{lemma}[{\cite[Lemma 8.7, Proposition 8.12]{QR24}}]\label{classification}
Let $\alpha$ be a $\qq$--ray, and let $M, c$ and $\rho$ be as in Proposition~\ref{sis12}.
Then either:
\begin{itemize}
\item $\alpha$ is a $\qq$--transient ray, then all quasi-geodesic rays in $\bfa =[\alpha]$ are transient. The class $\bfa$ has a geodesic representative $\alpha_0$, and every $\qq'$-ray in $\bfa$ can be $f_\bfa(\qq')$-quasi-redirected to $\alpha_0$, where \[f_\bfa(q, Q) = (9q, Q).\]
\item  Otherwise, $\alpha$ is not transient, then $\alpha$ is eventually contained in $N_{\rho M}(A)$ 
for some $A \in \calA$. Likewise all quasi-geodesic rays in $[\alpha]$ are non-transient, and all $\qq'$-rays in $[\alpha]$ are eventually contained in $N_{\rho(\qq')M(\qq')}(A)$ for the same $A$.
\end{itemize}
Furthermore, if $\alpha$ is a $\qq$--transient ray and $\qq' \geq \qq$, then $\alpha$ is also a 
$\qq'$--transient ray. 
\end{lemma} 

We remark  that $K = K(G)$ is a proper hyperbolic space on which $G$ acts properly discontinuously, and this action is geometrically finite. Every limit point of $\partial K$ is either a conical limit point or a bounded parabolic point~\cite{Bow12}. In particular, a limit point is a \emph{conical limit point} if the associated geodesic ray is a $(1,0)$-transient ray.

\subsection{Bowditch boundary}

Now we define the Bowditch boundary for relatively hyperbolic groups. Recall let $K$ be the coned-off Cayley graph introduced in the definition of relatively hyperbolic groups. Let $\partial K$ denote the Gromov boundary of $K$.  Let $V(K)$ denote the vertex set of $K$, 
 let $V_\infty K=  \{  *p_A \mid A \in \calA\}$ and let $\triangle K = V_\infty (K) \cup \partial K$. 
 
 \begin{defn}
 For $v, w \in V(K) \cup \partial K$, let $[v, w]_K$ denote a geodesic segment (or a geodesic ray) in $K$ 
 connecting $v$ to $w$. Given any $v \in V(K) \cup \partial K$ and  a finite set $W\subseteq V(K)$, we write
\[ 
m(v, W) = \Big\{ w \in \triangle K \,\Big\vert \,  W \cap [v, w]_K  \subseteq \{v \} \, \text{for every geodesic $[v, w]_K$} \Big\}.
\]
The \textit{Bowditch boundary} $\partial_B G$ of the relatively hyperbolic group $G$ is the set $\triangle K$
equipped with a topology generated by the neighborhoods of the form $m(v, W)$. 
\end{defn}

Every geodesic ray or segment in $K$ can be associated to some quasi-geodesic in $X=\Cay(G)$ as follows. 
Let $\ell$ be a path in $K$, a \emph{lift} of $\ell$, denoted $\overline{\ell}$, is a path formed from $\ell$ by replacing edges incident to vertices in $V_{\infty}(K)$ with a geodesic in $\Cay(G)$.

\begin{lemma}[{\cite[Lemma 9.4]{QR24}}]\label{ellgeodesic}
	There exists a uniform bound $D$ such that the following holds.
Let $\ell$ be a geodesic line or segment in $K$ such that $|\ell| \geq 3$. Then there exists a geodesic line $\overline{\ell}_0$ in $\Cay(G)$ such that the projection of $\overline{\ell}_0$ to $K$ is contained in the $D$-neighborhood  of $\ell$ in $K$.
\end{lemma}

We also recall the \emph{relative thin triangle} property  geodesic triangles in $\Cay(G)$~\cite[Theorem 1.1]{Sis12}:
\begin{proposition}[{\cite[Definition 3.11]{Sis12}}]\label{Prop:thin} 
There exists a constant $\delta_1$ such that the following holds. 
For points $x, y, z \in \Cay(G)$ consider a geodesic triangle $(x, y, z)$ and let $w$ be a 
$(1,0)$--transition point along $[x,y]$. Then there exists $w' \in [x,z] \cup [z,y]$ such that 
\[d_{\Cay(G)}(w, w') \leq \delta_1.\] 
\end{proposition} 



We now show that $(G, \calP)$ satisfies the assumptions associated to QR boundaries if the parabolic subgroups do. We first define the  \emph{shadow} of a non-transient quasi-geodesic into a parabolic subset $A$.
\begin{defn}\label{def:shadow}
 Let $\alpha$ be a $(q, Q)$-quasi-geodesic ray emanating from $\go$, such that $\alpha$ is non-transient. By Lemma~\ref{classification}, all but a finite segment $\alpha|_{[0,t_0]}$ of $\alpha$ is contained in $N_{\rho M}(A)$. Define   $\Sh_A(\alpha)$ by composing $\alpha|_{[t_0, \infty)}$ 
 with the closest-point projection to $A$, and by~\cite[Lemma 2.3]{QR24} the resulting map can be tamed to be a $(q', Q')$-quasi-geodesic that is also $2(q + Q)$--Lipschitz and fellow travels $\alpha$. We call this 
 $(q', Q')$-quasi-geodesic the \emph{shadow} of $\alpha$ in $A$, and we write it as $\Sh_A(\alpha)$. 
\end{defn}

\begin{theorem}\label{assumps}
If the QR boundaries exist for each subgroup $P\in \calP$, then the QR boundary of $(G, \calP)$ exists.
\end{theorem}
\begin{proof}
By~\cite[Lemma 2.5]{QR24}, any metric space quasi-isometric to an infinite  finitely generated group satisfies QR-Assumption 0. For QR-Assumption 1, it was shown in Lemma~\ref{classification} that all transient classes have a geodesic ray with a redirecting function 
\[f_\bfa (q, Q) = (9q, Q).\]
Now we address the case where a quasi-redirecting equivalence class $[\alpha]$  is non-transient.  

Since there are only finitely many elements of $\cP$ and every such subgroup satisfies QR-Assumption 1, there exists a $\qq_0$ such that for every $A\in \calA$, every element of $\partial A$ can be represented by a central $\qq_0$-ray $\beta_0$. 
We first claim is that for every  such  $\beta_0$, there is a $\qq_1$-ray in $X$, starting at the basepoint $\go$, which eventually coincides with $\beta_0$, where $\qq_1=\qq_1(\qq_0)$. Indeed, let $\beta_0$ be a central element in  $A$ for some $A\in \calA$, and consider a nearest point projection from $\go$ to $\beta_0$. Surgery Lemma~\ref{concate} implies that there exists a $\qq_1$-ray that starts from $\go$ and eventually coincides with  $\beta_0$, where $\qq_1=\qq_1(\qq_0)$.

Now consider any non-transient $\qq$-ray $\alpha$ in $X$. By Lemma~\ref{classification}, $\alpha$ is eventually in the bounded neighborhood of  some $A \in \calA$. 
By Definition~\ref{def:shadow}, $\Sh_A(\alpha)$ is a $\qq'$-ray in $A$ for some $\qq'=\qq'(\qq)$, and $\alpha \sim \Sh_A(\alpha)$. Since $A$ satisfies QR-Assumption 1, there is a central $\qq_0$-ray $\alpha_0$ in $[\Sh_A(\alpha)]$ and a redirecting function $f_{[\alpha_0]}$ such that $\Sh_A(\alpha)$ can be  $f_{[\alpha_0]}(\qq')$-quasi-redirected to $\alpha_0$. By the previous paragraph, there is a $\qq_1$-ray  emanating from $\go$, denoted $(\alpha_0)_\go$, which eventually coincides with $\alpha_0$. Therefore, $\Sh(\alpha)$ can be $f_{[\alpha_0]}(\qq')$-quasi-redirected to  $(\alpha_0)_\go$. Since $\alpha \sim \Sh(\alpha)$, there is a redirecting function $f_{[\alpha]}$ such that every non-transient $\qq$-ray $\alpha$ in $X$ can be $f_{[\alpha]}(\qq)$-quasi-redirected to $(\alpha_0)_\go$, where $f_{[\alpha]}$ depends on $f_{[\alpha_0]}$ and the constants in Definition~\ref{def:shadow}, 
 the Transitivity Lemma (Lemma~\ref{lem:transitive}) and the Surgery Lemma~\ref{lem:surgery1}.  Thus QR-Assumptions~1 and 2 are satisfied for non-transient rays.
Combining both cases, we see that all three assumptions are always satisfied. 
\end{proof}

\begin{defn}\label{defxi}
 We define a map 
\[
\xi: \partial G \to \partial_B G
\] 
as follows. Let $\bfa \in \partial G$ and  $\alpha_0 \in \bfa$ be the central element of $\bfa$. 
If $\alpha_0$ is not transient, then by Lemma~\ref{classification}, there exists some $A \in \calA$ such that a tail of $\alpha_0$ is in a bounded neighborhood of $A$.  In this case we define \[\xi(\bfa) \coloneqq *p_A.\]

Otherwise, $\alpha_0$ is transient. By the construction and hyperbolicity of $K$, $\alpha_0$ is an unbounded unparameterized quasi-geodesic in $K$ 
and hence converges to a point $\hat{\alpha}_0$ in $\partial K$. We define
\[\xi(\bfa) \coloneqq \hat{\alpha}_0. \]
\end{defn}

\begin{lemma}
The map $\xi: \, \partial G \to \partial_B G $ is surjective. 
\end{lemma}

\begin{proof}
 Let $v \in V_\infty (K)$ be a point in the Bowditch boundary and let $A$ be the associated set in $\calA$. Let $\alpha$ be a quasi-geodesic ray that connects $[\go, \go_A]$ with a geodesic ray starting at  $\go_A$ and lie entirely in $A$. By~\cite[Lemma 4.19]{DS05} $\alpha$ is a bounded constant quasi-geodesic ray in the class of $\partial A$. It follows that $\xi([\alpha]) = v$. 
 
 Otherwise, let $v$ be a point in $\partial K$. Since $K$ is hyperbolic, there exists an equivalence class of quasi-geodesic rays associated with $v$ and in fact there exists a geodesic 
representative in this class (for instance by the Arzel\'a--Ascoli  Theorem), which we refer to as $\alpha$. Since $\alpha$ is a geodesic ray in $K$, by~\cite[Proposition 1.14]{Sis13}, there
exists a bounded constant quasi-geodesic ray  $\alpha'$ in $\Cay(G)$ that is a lift of $\alpha$. 
We claim that, for $\bfa=[\alpha']$, we have 
\[ \xi(\bfa) = v. \]
Indeed, the central element $\alpha_0$ of $\bfa$ is a geodesic in $\Cay(G)$, and an unparameterized  
quasi-geodesic in $K$. Thus it stays in a bounded neighborhood of $\alpha$ and hence 
converges to $v$. This finishes the proof. 
 \end{proof}

We now show that $\xi$ and $\xi^{-1}$ are both continuous. First we show that for every $v \in \Delta(K)$ 
and every finite subset $W \subset V(K)$, $m(v, W)$ is open in $\partial G$. It suffices to verify this when 
$W$ has one element as a finite intersection of open sets is open.

\begin{lemma}\label{516}
For every $\bfb \in \partial G$ and $p \in V(K)$ there exists $r>0$ such that  
\[
\xi(\calU(\bfb, r)) \subset  m(\xi(\bfb), p). 
\]
Therefore, $\xi$ is continuous. 
\end{lemma}

\begin{proof}
Let $\beta_0$ be the central element of $\bfb$.

\noindent Case~I\@: We first assume that  $\bfb$ is transient. 
Consider $\beta_0$ as a subset of $K$ and let $\pi_{\xi(\bfb)}(p)$ be the closest point projection 
of $p$ to $\beta_0$ in $K$ (see Figure~\ref{Fig:quad}). 
Since $K$ is hyperbolic, $\pi_{\xi(\bfb)}(p)$ has bounded diameter in $K$. 
As $\bfb$ is transient, $\beta_0$ has transition points that are arbitrarily far from $\go$. 
Choose $r>0$ such that, $(\beta_0)_r$ is a $(1,0)$--transition point of $\beta_0$ and 
\begin{equation}\label{case91}
d_K(\go, (\beta_0)_r) \gg d_K(\go, \pi_{\xi(\bfb)}(p))+D(9,0)+2\delta,
\end{equation}
where $\delta$ is the hyperbolicity constant of $K$, $D(9,0)$ is as in~\cite[Corollary 8.8]{QY24}, and 
$d_K(\go, \pi_{\xi(\bfb)}(p))$ is the maximum distance in $K$ between any point in $\pi_{\xi(\bfb)}(p)$ to $\go$. 

Let $\bfa \in  \calU(\bfb, r)$ and let $\alpha_0$ be the central element in $\bfa$. Since $(\beta_0)_r$ is a transition 
point, there exists a point $q \in \alpha_0$ such that  
 \[
 d(q, (\beta_0)_r) < D(9,0).
 \]
 Thus $\Norm{q} \geq r - D(9,0)$. As $K$ is hyperbolic, there exists  a geodesic $\ell$ in $K$ connecting
$\xi(\bfa)$ to $\xi(\bfb)$. The line $\ell$ is an edge in the ideal quadrilateral $((\beta_0)_r, \xi(\bfb), \xi(\bfa), q)$
hence it stays in a bounded neighborhood of 
\[
\beta_0|_{\geq r} \cup \alpha_0|_{\geq r} \cup [(\beta_0)_r, q].
\] 
Thus $\ell$ is far from $p$ in $K$, and hence does not pass through $p$. Therefore, 
$\xi(\bfa) \in m(\xi(\bfb), p)$.

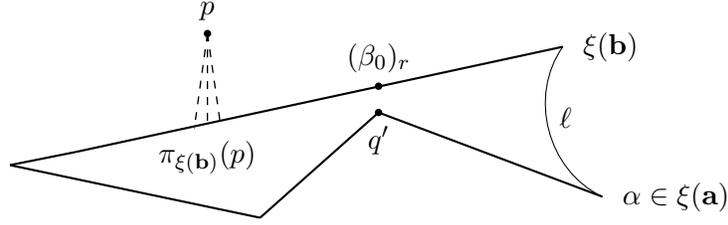
\begin{figure}
\centering
\begin{tikzpicture}[scale=0.7]
\tikzstyle{vertex} =[circle,draw,fill=black,thick, inner sep=0pt, minimum size=2pt] 
\node[vertex] at (4.5,9)[label=above:$p$]{};

\draw[thick] (2,6.75) -- (11.25,8.75);
\draw[thick] (2,6.75) -- (0.75,6.5);
\draw [line width=0.5pt, dashed] (4.5,9) -- (4.25,7.2);
\draw [line width=0.5pt, dashed] (4.5,9) -- (4.5,7.25);
\draw [line width=0.5pt, dashed] (4.5,9) -- (4.75,7.25);
\node at (4.5,7.35) [label=below:$\pi_{\xi(\bfb)}(p)$]{};
\draw [thick](0.75,6.5) -- (5.5,5.5);
\draw[thick] (5.5,5.5) -- (7.75,7.5);
\node at  (11.25,8.75)[label=right:$\xi(\bfb)$]{};
\node[vertex] at  (7.75,8)[label=above:$(\beta_0)_r$]{};
\node[vertex] at  (7.75,7.5)[label=below:$q'$]{};
\draw[thick] (7.75,7.5) -- (12,5.9);
\node at (12,5.9)[label=right:$\alpha \in \xi(\bfa)$]{};
 \draw(12,5.9) to[bend left=50](11.25,8.75){};
\node at (10.8, 7.4) [label=right:$\ell$]{};
\end{tikzpicture}
\caption{A transition point $(\beta_0)_r$ separates the point $p$ and any geodesic line that connects $\xi(\bfb)$ and $\xi(\bfa)$.}\label{Fig:quad}
\end{figure}

\noindent Case II\@: Suppose  that $\bfb$ is not transient.  By Lemma~\ref{classification}, there exists a unique set $A \in \calA$ such that $\xi(\bfb) = *p_A$.  Let $\beta_0$ be the central element of $\bfb$. Let 
\[
r \gg  2\big(\Norm{\go_A}+ \Norm{p} \big).
\]
Let $\bfa \in \calU(\bfb, r)$ and let $\alpha_0$ be the central element of $\bfa$. Then $\alpha_0$ can be 
$f_\bfb(1,0)$--quasi-redirected to $\beta_0$ at radius $r$. Let $\gothic{e} \in A$ be the point where $\alpha_0$ 
leaves the $M_0$--neighborhood of $A$, where $M_0\coloneqq M(1,0)$ is as in Proposition~\ref{sis12}.

%
 
 Consider any geodesic segment or ray $\ell$ in $K$ connecting $\xi(\bfa)$ to $*p_A$. By~\cite[Proposition 8.13]{Hru10}, 
 $\ell$ enters $N_{\tau(f_\bfb(1,0))}(A)$ at a point that is boundedly close to $\gothic{e}$. Since $*p_A$ is an endpoint 
 of $\ell$, $*p_A$ does not appear in interior of $\ell$ and hence, for any other vertex $x$ in $\ell$, we have 
 $\Norm{x} \geq \Norm{\gothic{e}} - D(1,0)$. This implies $\Norm{x} \gg \Norm{p}$ and hence $\ell$ does not pass through $p$. 
 Therefore, 
 \[
 \bfa \in m(\xi(\bfb), p)
 \]
 and hence $\calU(\bfb, r) \subset m(\xi(\bfb), p)$. 
 \end{proof}


Now we are ready to conclude:
\begin{theorem}\label{surjective}
Let  $G$ be a relatively hyperbolic group with respect to subgroups $P_1, P_2,\dots,P_k$. Assume that $\partial P$ exists for each Cayley graph of the subgroups $P \in \calP$, then the quasi-redirecting boundary $\partial G$ exists and $\partial G$ surjects onto $\partial_B G$.
\end{theorem}
\begin{proof}

Since the map $\xi : \, \partial X \to \partial_B X$ is onto and $\xi$ is continuous, we conclude that $\xi : \, \partial G \to \partial_B G$ is a surjective homomorphism.
\end{proof}
\begin{cor}
Let  $G$ be a relatively hyperbolic group with respect to subgroups $P_1, P_2,\dots,P_k$. Then the conical limit points of $K$ are embedded as a subset in $P(G)$.
\end{cor}
\begin{proof}
Case I of Lemma~\ref{516} shows that if $\bfb$ has a transient geodesic ray representative then it maps to exactly one point in $\partial K$. Therefore there is a 1-1 map between the set of conical limit points of $G$ and the set of  transient classes in $P(G)$.
\end{proof}

\section{Quasi-geodesics in templates}
\label{sec:template}
\subsection{Templates}
In this section, we will revisit the concept of \textit{templates} introduced in \cite{CK02} and study its quasi-redirecting boundary. Roughly speaking, templates are essentially piecewise Euclidean Hadamard spaces that can be embedded in $\R^3$. They approximate certain subspaces of the spaces we are studying and contain a great deal of information about the spaces at infinity. Our analysis of quasi-redirecting for quasi-geodesics in templates will serve as a foundation for studying quasi-redirecting in quasi-geodesics of Croke-Kleiner admissible groups in the subsequence section.


\begin{defn}
 \label{defn:template}
A \textit{template} is a connected Hadamard space  $\mathcal{T}$ (indeed piecewise Euclidean) obtained from the disjoint collection of Euclidean planes $\{F\}_{F \in \operatorname{Wall}_{\mathcal{T}}}$ (called \textit{walls}) and  Euclidean strips $\{ S \simeq I \times \mathbb R\}_{\mathcal{S} \in \operatorname{Strip}_{\mathcal{T}}}$ (where $I$ is a closed interval of $\mathbb R$) by isometric gluing subject to the following conditions.
\begin{enumerate}
    \item The boundary geodesics of each $S \in \operatorname{Strip}_{\mathcal{T}}$, which we will refer to as \textit{singular geodesics}, are glued isometrically to distinct walls in $\operatorname{Wall}_{\mathcal{T}}$.
    \item Each wall $F \in \operatorname{Wall}_{\mathcal{T}}$ is glued to at most two strips, and the gluing lines are not parallel.
 \end{enumerate}   
\end{defn}

A template $\mathcal{T}$ can be visualized in $\mathbb{R}^3$ with its walls as parallel planes and its strips meeting the walls orthogonally.

\begin{itemize}
    \item Two walls $F_1, F_2 \in \operatorname{Wall}_{\mathcal{T}}$ are \textit{adjacent} if there is a strip $S \in \operatorname{Strip}_{\mathcal{T}}$ such that $S \cap F_1 \neq \emptyset$ and $S \cap F_2 \neq \emptyset$.

    \item  A wall
is an \textit{interior wall} if it is incident to two strips, and a strip is an \textit{interior strip}
if it is incident to two interior walls. The sets of interior walls and strips are denoted $\operatorname{Wall}^{0}_{\mathcal{T}}$ and $\operatorname{Strip}^{0}_{\mathcal{T}}$  respectively.

\item  There is an associated \textit{angle function}
$$
\theta \colon \operatorname{Wall}^{0}_{\mathcal{T}} \to (0, \pi)
$$ that assigns to each interior wall the angle between the oriented singular geodesics $F \cap S_1$, $F \cap S_2$ where $S_1$ and $S_2$ are the two strips incident to $F$.
\end{itemize}


\subsection{Backward spiral paths in templates}
\label{subsec:backwardspiral}
In this subsection, let  
\[
\mathcal{T} := F_0 \cup S_{01} \cup F_1 \cup S_{12} \cup \dots \cup F_n
\]  
be a template in $\mathbb{R}^3$ as defined in Definition~\ref{defn:template}, with a constant angle function $\theta \equiv \pi/2$. That is, every pair of singular geodesics in the same wall meet at a right angle. We refer to such a template as a \textit{right-angled template} throughout this paper.  

For notation, we set:  
\begin{enumerate}
    \item A fixed basepoint $\gothic{o} \in F_0$.  
    \item The intersection of two adjacent strips:  
    \[
    p_i := S_{(i-1)i} \cap S_{i(i+1)}.
    \]  
    \item For each $i \geq 1$, the two singular geodesics in the wall $F_i$:  
    \[
    f_i^- := F_i \cap S_{(i-1)i}, \quad f_i^+ := F_i \cap S_{i(i+1)}.
    \]  
\end{enumerate}

To understand how quasi-geodesics behave in templates, we introduce two fundamental types of paths: $L$--paths, which stay inside a wall, and $Z$--paths, which cross between strips and walls. These paths will be key building blocks in our construction of backward spiral paths.

\begin{defn}
For each $ i \geq 1 $, an \textit{$ L $-path} in a wall $ F_i $ of $ \mathcal{T} $ is a concatenation of two geodesics $ l $ and $ l' $ in $ F_i $, where $ l $ is parallel to the singular geodesic $ f_i^- $ and $ l' $ is parallel to $ f_i^+ $.  

A \textit{$ Z $-path} in $ \mathcal{T} $ consists of an $ L $-path in $ F_i $ followed by a geodesic segment in a strip adjacent to $ F_i $, perpendicular to the singular geodesic $ f_i^- $.

\end{defn} 

\begin{rem}
\label{rem:extended-L-path}
Since an $ L $-path consists of two perpendicular segments in the Euclidean plane $ \mathbb{E}^2 $, it is a $(\sqrt{2},0)$-quasi-geodesic. By the geometry of the template $ \mathcal{T} $ and Lemma~\ref{concate}, a $ Z $-path is a $(3\sqrt{2}, 0)$-quasi-geodesic.

\end{rem}

\subsection*{Construction of backward spiral path}
We construct a \textit{backward spiral path} by concatenating a sequence of $L$--paths and $Z$--paths, ensuring that each step moves deeper into the template in a controlled way.
 For a path $ \gamma $, we denote its initial and terminal points by $ \gamma_{-} $ and $ \gamma_{+} $, respectively.

Given $ q \ge 1, Q \ge 0, \delta \in (0,1] $, and a constant $ \rho > \frac{q}{\delta} + Q $, we construct the following paths:
\begin{enumerate}
\item Given $ x_n $ in the last wall $ F_n $, we attach  a $ Z $-path  
    \[
    Z_n \coloneqq v_n \cdot h_n \cdot \eta_{n-1},
    \]  
    where 
    \begin{itemize}
        \item $ \eta_{n-1} $ is a geodesic in the strip $ S_{(n-1)n} $ perpendicular to the singular geodesic $ f_n^- $.
        \item $ v_n $ is a geodesic segment based at $ x_n $ and parallel to the singular geodesic $ f_n^- $ such that $ \operatorname{Length}(v_n) \ge 2 q Q $.
        \item $ h_n $ is a geodesic segment in $ F_n $ parallel to $ f_n^+ $ with terminal point $ (h_n)_+ \in f_n^- $.
    \end{itemize}   Note that $ \operatorname{Length}(h_n) $ is the distance from $ x_n $ to $ f_n^- $.

\item  Repeat the process for each wall $F_{n-1}, F_{n-2}, \ldots, F_1$. At each step $i$,  we attach a $ Z $-path 
    \[
    Z_{n-1} \coloneqq v_{n-1} \cdot h_{n-1} \cdot \eta_{n-2},
    \]  to the terminal point $ (Z_i)_+ $ 
    where 
    \begin{itemize}
        \item $ v_{i-1} $ is a geodesic segment in $ F_{i-1} $ based at $ (Z_i)_+ $ and parallel to the singular geodesic $ f_{i-1}^- $;

        \item $ h_{i-1} $ is a geodesic segment in $ F_{i-1} $ parallel to $ f_{i-1}^+ $;

        \item  $ \eta_{n-2} $ is a geodesic in the strip $ S_{(n-2)(n-1)} $ perpendicular to $ f_{n-1}^- $. Additionally, we require  
    \[
    v_{n-1} > \rho \cdot \max \left\{ d((Z_n)_+, (Z_n)_-), h_{n-1} \right\}.
    \]
    \end{itemize} 

\item  We continue this pattern to define extended $ L $-paths:  
   \[
   Z_{n-2}, Z_{n-3}, \ldots, Z_i, \ldots, Z_1.
   \]

\item Terminate the sequence by attaching a final geodesic ray $Z_0$ in $F_0$. Unlike previous steps, this final segment is a straight geodesic without an additional $L$--path.

\end{enumerate}

The resulting concatenation:
$$\mathcal{Z}_{q, Q, \delta, \rho, x_n} : = Z_{n} \cdot Z_{n-1} \cdots Z_1$$ is called  a \textit{backward spiral path} (see Figure~\ref{backward}).

\begin{figure}[htb]
\centering 
 \def\svgwidth{0.5\textwidth}
    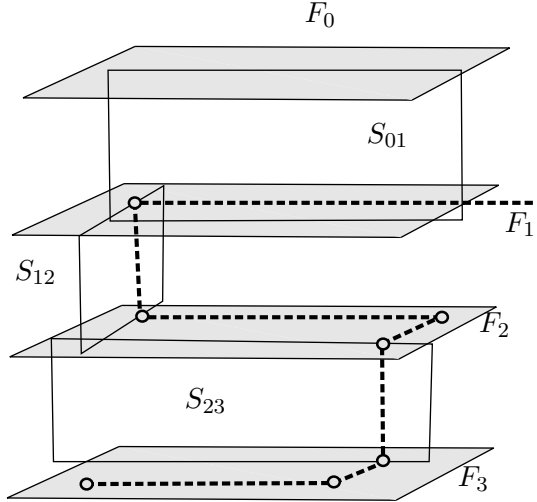
\caption{The figure illustrates a portion of a backward spiral path in a right-angled template}
\label{backward}
\end{figure}

\begin{rem}
\label{rem:extended-L-pathsarequasigeodesics}
    For each $ i $, since $ v_i \cdot h_i $ is a concatenation of two perpendicular segments in a plane, it is a $ (\sqrt{2}, 1) $-quasi-geodesic. Since $ \eta_{i-1} $ lies in the strip $ S_{(i-1)i} $ and is perpendicular to the singular geodesic $ f_i^- $, it follows that $ (h_i)_+ $ is the closest point on $ F_i $ to $ (\eta_{i-1})_+ $. Hence, $ Z_i $ is a $ (3\sqrt{2}, 1) $-quasi-geodesic by Lemma~\ref{concate}. Similarly, from the geometry of the template $ \mathcal{T} $, $ (\eta_{i-1})_+ $ is the closest point on $ S_{(i-1)i} \cup F_i $ to $ v_{i-1} $, and therefore $ Z_i \cdot v_{i-1} $ is a $ (9\sqrt{2}, 1) $-quasi-geodesic by Lemma~\ref{concate}.

According to the geometry of $ \mathcal{T} $, we have $ d((h_{i-1})_+, x) > \operatorname{Length}(v_{i-1}) $ for all $ x \in Z_i $. Recall that  
\[
\operatorname{Length}(v_{i-1}) > \rho \operatorname{max} \{ d((Z_i)_+, (Z_i)_-), h_{i-1} \}.
\]  
Thus, we have $ d((h_{i-1})_+, x) > \operatorname{Length}(h_{i-1}) $. As a result, $ (v_{i-1})_+ $ is the closest point on $ Z_i \cdot v_{i-1} $ to $ (h_{i-1})_+ $. By Lemma~\ref{concate}, $ Z_i \cdot v_{i-1} \cdot h_{i-1} $ is a $ (27\sqrt{2}, 1) $-quasi-geodesic. Similarly, $ Z_i \cdot Z_{i-1} = Z_i \cdot v_{i-1} \cdot h_{i-1} \cdot \eta_{i-2} $ is a $ (81\sqrt{2}, 1) $-quasi-geodesic.

For each point $ u \in Z_i $, we have $ d(u, Z_{i-1}) \ge \operatorname{Length}(v_{i-1}) $. Since $ \operatorname{Length}(v_{i-1}) > \rho \operatorname{max} \{ d((Z_i)_+, (Z_i)_-), h_{i-1} \} $, it follows that  
\[
d(u, Z_{i-1}) > \rho d((Z_i)_+, (Z_i)_-)
\]  
for all points $ u $ in $ Z_i $.
\end{rem}

\subsection*{Backward spiral paths are uniform quasi-geodesics}
We will now show that backward spiral paths are quasi-geodesics with uniform quasi-geodesic constants. To do so, we need the following result.

\begin{prop}
\label{prop:constructQg}
Let $ X $ be a metric space. Given constants $ q \ge 1 $, $ Q \ge 0 $, and $ \delta \in (0,1] $, and for every positive constant $ \rho > q/\delta + Q $, there exist uniform constants $ L = L(\rho, q, Q, \delta) $ and $ C = C(\rho, q, Q) $ such that the following property holds:

Let
\[
\gamma := \gamma_1 \cdot \gamma_2 \cdots \gamma_n
\]
be a concatenation of $ (q, Q) $-quasi-geodesics $ \gamma_i $, such that the following conditions are satisfied:
\begin{enumerate}
    \item \label{item1} $ d((\gamma_1)_-, (\gamma_1)_+) \ge 2qQ $.
    \item \label{item2} $ (\gamma_i)_+ = (\gamma_{i+1})_- $ for $ 1 \le i \le n-1 $.
    \item \label{item3} The concatenation $ \gamma_i \cdot \gamma_{i+1} $ is a $ (q, Q) $-quasi-geodesic for each $ 1 \le i \le n-1 $.
    \item \label{item4} $ d((\gamma_i)_-, (\gamma_i)_+) \ge \rho d((\gamma_{i-1})_-, (\gamma_{i-1})_+) $ for $ 2 \le i \le n $.
    \item \label{item5} For any point $ u \in \gamma_{i+1} $, we have $ d(u, (\gamma_i)_-) \ge \delta \, d((\gamma_i)_-, (\gamma_i)_+) $.
\end{enumerate}
Then $ \gamma $ is a $ (L, C) $-quasi-geodesic.

\end{prop}

\begin{proof}
Let $ 0 = a_0 < a_1 < a_2 < \cdots < a_n $ such that $ \gamma_i = \left. \gamma \right|_{[a_{i-1}, a_i]} $. Now, let $ t_1 $ and $ t_2 $ be distinct points in $ [a_0, a_n] $. To simplify the notation, we define:
\[
\abs{\gamma_i} := d\left(\gamma(a_{i-1}), \gamma(a_i)\right)
\]
for $ i = 1, 2, 3, \dots, n $.
\textbf{Claim 1:}  
\[
d\left(\gamma\left(t_1\right), \gamma\left(t_2\right)\right) \leq L (t_2 - t_1)+ C
\] where $L$ and $C$ defined explicitly in the proof.

We consider the case where $ t_1 \in [a_k, a_{k+1}] $ and $ t_2 \in [a_j, a_{j+1}] $ with $ j-k \geq 1 $. We have:
\[
t_2 - t_1 \geq a_j - a_{k+1} = \sum_{s=k+2}^{j} (a_s - a_{s-1})
\]

By (\ref{item4}), we know $ \abs{\gamma_s} \geq \rho \abs{\gamma_1} > Q+1 $, so we obtain:
\[
\abs{a_s - a_{s-1}} \geq \frac{\rho \abs{\gamma_1} - Q}{q} > \frac{1}{q}
\]
Thus:
\[
t_2 - t_1 \geq \sum_{s=k+2}^{j} (a_s - a_{s-1}) \geq \frac{j-k-1}{q}
\]

Next, by the triangle inequality and since $ a_{k+1} - a_j \leq 0 $, we have:
\begin{align*}
 d(\gamma(t_1), \gamma(t_2)) &\leq d(\gamma(t_1), \gamma(a_{k+1})) + \sum_{s=k+2}^{j} d(\gamma(a_s), \gamma(a_{s-1})) + d(\gamma(a_j), \gamma(t_2))   \\~\\
 &\le q (t_2 - t_1) + q \sum_{s=k+2}^{j} (a_{s} - a_{s-1}) + (j-k-1) Q + 2Q \\~\\
 &\leq q (t_2 - t_1) + q (t_2 - t_1) + qQ(t_2 -t_1) + 2Q \\~\\
 &= (2q + qQ) (t_2 -t_1) + 2Q
\end{align*}

\textbf{Claim~2:}$$d\left(\gamma\left(t_1\right), \gamma\left(t_2\right)\right) \ge (1/L) (t_2 -t_1) - C$$By (\ref{item3}), $\gamma_i \cdot \gamma_{i+1}$ is a $(q, Q)$--quasi-geodesic for every $i$, we only need to consider the case where $t_1 \in\left[a_k, a_{k+1}\right]$ and $t_2 \in\left[a_j, a_{j+1}\right]$ with $j \geq k+ 2$. By the triangle inequality,
\begin{equation}
\tag{$\diamondsuit$}
\label{equation1}
d\left(\gamma\left(t_2\right), \gamma\left(t_1\right)\right) \geq d\left(\gamma\left(t_2\right), \gamma\left(a_{j-1}\right)\right)-d\left(\gamma\left(a_{j-1}\right), \gamma\left(t_1\right)\right) 
\end{equation}
By (\ref{item4}), we have

\begin{equation}
\tag{$\clubsuit$}
\label{equation2}\sum_{i=1}^{j-1} \abs{\gamma_i} \le \frac{1}{\rho-1}\abs{\gamma_j} \le \frac{2}{\rho}\abs{\gamma_j}
\end{equation}

From the triangle inequality, we have:
\begin{align*}
 d\left(\gamma\left(t_1\right), \gamma\left(a_{j-1}\right)\right) &\leq d\left(\gamma\left(t_1\right), \gamma\left(a_{k+1}\right)\right) + \sum_{i=k+2}^{j-1} \abs{\gamma_i}   \\~\\
 &\leq q ( a_{k+1} - a_k ) + Q + \sum_{i=k+2}^{j-1} \abs{\gamma_i} \\~\\
 & \leq q \left( \sum_{i=k+1}^{j-1} \abs{\gamma_i} \right) + q Q + Q
\end{align*}

This expresses the upper bound on the distance $ d(\gamma(t_1), \gamma(a_{j-1})) $ in terms of the lengths of the segments $ \gamma_i $.


By applying (\ref{item5}), we have:

\[
d(\gamma(t_2), \gamma(a_{j-1})) \ge \delta \, d(\gamma(a_j), \gamma(a_{j-1})) = \delta \, \abs{\gamma_j}
\]

Then, using inequality (\ref{equation2}), we get:
\begin{align*}
 d(\gamma(t_1), \gamma(a_{j-1})) &\le q \left(\sum_{i=k+1}^{j-1} \abs{\gamma_i} \right) + qQ + Q    \\~\\
 &\le \frac{2q}{\rho} \abs{\gamma_j} + qQ + Q \\~\\
 &\le \frac{2q}{\delta \rho} d(\gamma(t_2), \gamma(a_{j-1})) + qQ + Q
\end{align*}

This provides an upper bound for $ d(\gamma(t_1), \gamma(a_{j-1})) $ in terms of $ d(\gamma(t_2), \gamma(a_{j-1})) $.

By substituting the previous result into inequality (\ref{equation1}) and applying the fact from (\ref{item3}) that $ \gamma_s \cdot \gamma_{s+1} $ is a $(q, Q)$-quasi-geodesic, we obtain:
\begin{align*}
   d\left(\gamma(t_2), \gamma(t_1)\right) &\ge d\left(\gamma(t_2), \gamma(a_{j-1})\right) - d\left(\gamma(a_{j-1}), \gamma(t_1)\right)  \\~\\
   &\ge (1 - \frac{2q}{\delta \rho}) d (\gamma(t_2), \gamma(a_{j-1})) - qQ - Q \\~\\
   &\ge \frac{\delta \rho - 2q}{\delta \rho} \left( \frac{1}{q} \left|t_2 - a_{j-1}\right| - Q \right) - qQ - Q \\~\\
   & = \frac{\delta \rho - 2q}{\delta \rho} \frac{1}{q} \left|t_2 - a_{j-1}\right| - \frac{(\delta \rho - 2q) Q}{\delta \rho} - qQ - Q \\~\\
   & \ge c_1 \left|t_2 - a_{j-1}\right| - c_2
\end{align*}
where $ c_1 = c_1(\rho, q) $ and $ c_2 = c_2 (\rho, q, Q) $. We note that $ c_1 > 0 $ because $ \rho > \frac{2q}{\delta} $.

Recall that
\[
a_s - a_{s-1} \geq \frac{\rho |\gamma_1| - Q}{q} \geq 2qQ,
\]
which implies
\[
|\gamma_s| \geq \frac{1}{q} |a_s - a_{s-1}| - Q \geq \frac{1}{2q} |a_s - a_{s-1}|.
\]

From inequality (\ref{equation2}) and the fact that $ \gamma_j $ is a $ (q, Q) $-quasi-geodesic, we obtain:
\[
\frac{2}{\rho} \left( q |a_j - a_{j-1}| + Q \right) \geq \frac{2}{\rho} |\gamma_j| \geq \sum_{i=1}^{j-1} |\gamma_i| \geq \frac{1}{2q} \sum_{i=1}^{j-1} |a_i - a_{i-1}|.
\]
This simplifies to
\[
\frac{1}{2q} a_{j-1} \geq \frac{1}{2q} (a_{j-1} - t_1),
\]
which gives the bound
\[
a_{j-1} - t_1 \leq c_3 (a_j - a_{j-1}),
\]
for some constant $ c_3 = c_3 (\rho, q, Q) $.

Next, we have:
\[
(t_2 - t_1) = (t_2 - a_j) + (a_j - a_{j-1}) + (a_{j-1} - t_1),
\]
which is bounded by
\[
(t_2 - a_{j-1}) \left( 2 + c_3 \right),
\]
leading to
\[
(t_2 - a_{j-1}) \geq \frac{1}{2 + c_3} (t_2 - t_1).
\]

Thus,
\[
d(\gamma(t_2), \gamma(t_1)) \geq \frac{1}{L} (t_2 - t_1) - C,
\]
where $ L = \frac{2 + c_3}{c_1} $ and $ C = c_3 $.

By adjusting $ L $ and $ C $ if necessary, we conclude that $ \gamma $ is an $ (L, C) $-quasi-geodesic, as shown in Claim 1 and Claim 2.

\end{proof}

\begin{cor}
\label{cor:backwardlogisquasigeodesic}
    Let $ q = 81 \sqrt{2} $, $ Q = 1 $, $ \delta \in (0,1] $, and $ \rho > \frac{q}{\delta} + Q $. For $ x_n \in F_n $, the backward spiral path $ \mathcal{Z}_{q, Q, \delta, \rho, x_n} $ is an $ (L, C) $-quasi-geodesic, where $ L = L(q, Q, \delta, \rho) $ and $ C = C(q, Q, \delta, \rho) $.
\end{cor}

\begin{proof}
    By Remark~\ref{rem:extended-L-pathsarequasigeodesics}, the backward spiral path $ \mathcal{Z}_{q, Q, \delta, \rho, x_n} $ satisfies conditions (\ref{item1}), (\ref{item2}), (\ref{item3}), (\ref{item4}), and (\ref{item5}) in Proposition~\ref{prop:constructQg}. Hence, $ \mathcal{Z}_{q, Q, \delta, \rho, x_n} $ is an $ (L, C) $-quasi-geodesic, with constants $ L $ and $ C $ given by Proposition~\ref{prop:constructQg}.
\end{proof}

\subsection{Forward spiral paths in templates: Type I}
\label{subsection:logspiralTypeI}
In this section, we construct forward spiral paths of Type I, which play a crucial role in understanding the quasi-geodesic structure in right-angled templates. These paths exhibit a controlled growth pattern, ensuring that they satisfy the quasi-geodesic property.

Let $$ \mathcal{T} := F_0 \cup S_{01} \cup F_1 \cup S_{12} \cup \ldots \cup F_n $$ be a right-angled template in $ \mathbb{R}^3 $, as defined in Definition~\ref{defn:template}. Recall  for each $ 1 \le i \le n-1 $, we define $ p_i $ as the intersection point of the two singular geodesics $ f_{i}^{-} $ and $ f_{i}^{+} $ in $ F_i $.

\begin{lem} \cite[Lemma 3.3]{NY23}
There exists a uniform constant $ \mu \ge 1 $ such that for each $ n \ge 2 $, let $ \mathcal{T} := F_0 \cup S_{01} \cup F_1 \cup S_{12} \cup \ldots \cup F_n $ be a right-angled template. Then the concatenation $ [p_1, p_2] \cdots [p_{n-2}, p_{n-1}] $ is a $ (\mu, \mu) $-quasi-geodesic.

\end{lem}

\subsection*{Construction of forward spiral paths in templates: Type I}
Given a sufficiently small $ \rho_0 \in (0, 1/4) $, suppose there exists $ C > 0 $ such that for all $ i $,
$$ w_i := d(p_i, p_{i+1}) \le C (1 + \rho_0)^i. $$

\begin{enumerate}
    \item (Choosing geodesic rays). For each $i$, we define two geodesic rays:
    \begin{itemize}
        \item $ \ell_i^+ \subset f_i^+ $ based at $ p_i $.
        \item $ \ell_{i+1}^- \subset f_{i+1}^- $ based at $ p_{i+1} $.
    \end{itemize}
    We ensure that the projection of $ p_{i+1} $ onto $ f_i^+ $ lies on $ \ell_i^+ $. Additionally, we require that $ \ell_i^+ $ and $ \ell_{i+1}^- $ fellow-travel (see Figure~\ref{choiceofelli}).
    \item (Defining the forward spiral segments) For every $r > C$ we define
\[
\kappa_i : = r (1 + \rho_0)^i
\]
On $\ell_{i}^{+}$, choose $z_i$ so that $d(z_i, p_i) = \kappa_i$.
Let the width of the strip $ \mathcal{S}_{i(i+1)} $ be denoted by $ \delta_i $.

 Let $ y_i $ be the projection of $ z_i $ onto $ f_{i+1}^- $. 
 Since $ w_i < \kappa_i $, we have 
 \begin{itemize}
     \item $ y_i \in \ell_{i+1}^- $,
     \item $ d(y_i, z_i) = \delta_i $,
     \item $ d(y_i, p_{i+1}) \leq \kappa_i $.
 \end{itemize} 

\item (Concatenating the path)
The \textit{forward spiral path of Type I to $F_k$}, with $k < n$, denoted by $ L_{r,k} $, is the path obtained by concatenating the following segments:
\[
L_{r,k}(\alpha) \coloneqq \zeta|_{[0, r]} \cup [\zeta(r), z_1] \cup [z_1, y_1] \cup [y_1, z_2] \cup \dots \cup [z_{k-1}, y_{k-1}]
\] where 
\begin{itemize}
    \item $ \zeta $ is the given main flat ray.
    \item   $ \{ z_i \}, \{ y_i \} $ represent the intermediate points in the construction of the path (see Figure~\ref{template}) as in the prervious step.
\end{itemize}
\end{enumerate}

\begin{figure}[htb]
\centering 
 \def\svgwidth{0.5\textwidth}
 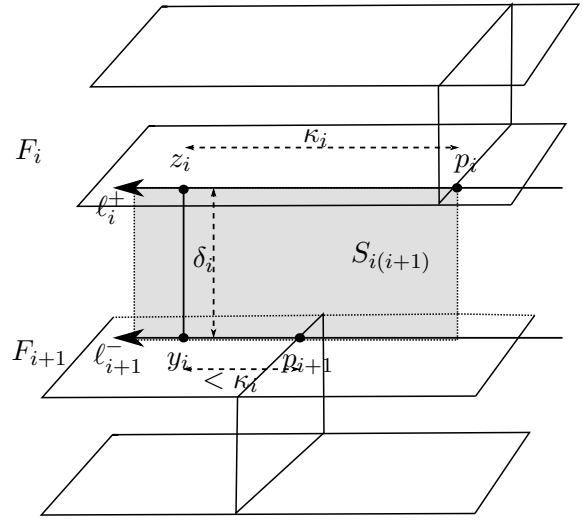
\caption{The figure illustrates how we choose geodesic rays $\ell_{i}^{+}$ and $\ell_{i+1}^{-}$ on the strip. Our choice of constant $\kappa_i > w_i = d(p_i, p_{i+1})$ ensures that the projection point $y_i$ of $z_i$ into $f_{i+1}^{-}$ will lie in $\ell_{i+1}^{-}$ and $d(y_i, p_{i+1}) \le \kappa_i$.}\label{choiceofelli}
\end{figure}

\subsection*{Forward spiral paths are uniform quasi-geodesics}
We note that if $i < j < n$ then $L_{r,i}$ is a subpath of $L_{r, j}$.

\begin{figure}[htb]
\centering 
 \def\svgwidth{0.5\textwidth}
 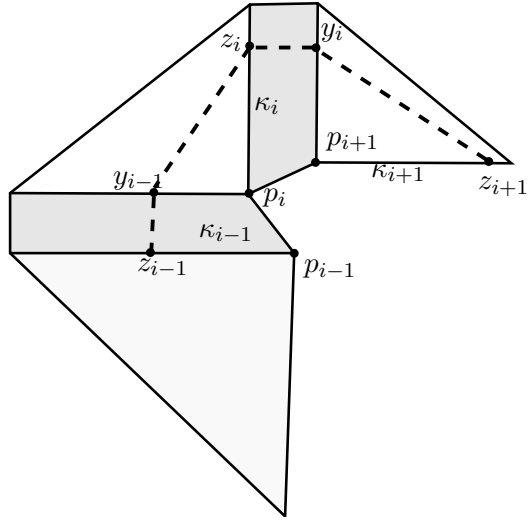
\caption{The figure illustrates a portion of $L_{r, k}$ which is a concatenation of dashed segments. The sum of all dashed segments is bounded above by an exponential function $(1+\rho_0)^k$ up to some multiplicative constant.}\label{template}
\end{figure}


Since $r > C$ we have

\[
\sum_{i=1}^{k-1} w_i < r \sum_{i=1}^{k-1} (1 + \rho_0)^i = r \frac{1 + \rho_0}{\rho_0} \bigl( (1 + \rho_0)^{k-1} - 1 \bigr),
\]
and
\[
\sum_{i=1}^{k} \kappa_i = \sum_{i=1}^{k} r (1 + \rho_0)^i = r \frac{1 + \rho_0}{\rho_0} \bigl( (1 + \rho_0)^k - 1 \bigr) \leq \frac{r (1 + \rho_0)^{k+1}}{\rho_0}.
\]

From our construction, we have $ d(z_i, y_i) = \delta_i \leq w_i $ and $ d(y_i, z_{i+1}) \leq d(y_i, p_{i+1}) + d(p_{i+1}, z_{i+1}) \leq \kappa_i + \kappa_{i+1} $.

For $ i < j $, let us denote the subpath of $ L_{r,k} $ from $ y_{i-1} $ to $ z_j $ by $ L_{r,k}|_{[y_{i-1}, z_j]} $. We then have:

\[
\operatorname{Length} \left( L_{r,k}|_{[y_{i-1}, z_j]} \right) \leq d(y_{i-1}, z_i) + d(z_i, y_i) + \cdots + d(z_{j-1}, y_{j-1}) + d(y_{j-1}, z_j).
\]

This simplifies to:
\begin{align*}
    \operatorname{Length} \left( L_{r,k}|_{[y_{i-1}, z_j]} \right) &\le 2 \sum_{m=i-1}^{j} \kappa_m + 2 \sum_{m=i-1}^{j-1} w_m \\~\\
    &\le 4r \sum_{m=i-1}^{j} (1 + \rho_0)^m \leq \frac{4r}{\rho_0} (1 + \rho_0)^{j+1}.
\end{align*}

By~\cite[Proposition~3.8]{NY23}, the subpath $ [p_i, p_{i+1}] \cdots [p_{j-1}, p_j] $  is a $ (\mu, \mu) $-quasi-geodesic. By Lemma~\ref{concate}, the concatenation

\[
\sigma \coloneqq [y_{i-1}, p_i] \cdot [p_i, p_{i+1}] \cdots [p_{j-1}, p_j] \cdot [p_j, z_j]
\]

is a $ (9\mu, 9\mu) $-quasi-geodesic. Specifically, $ p_i $ is the closest point on $ [p_i, p_{i+1}] \cdots [p_{j-1}, p_j] $, and thus $ [y_{i-1}, p_i] \cdot [p_i, p_{i+1}] \cdots [p_{j-1}, p_j] $ is a $ (3\mu, 3\mu) $-quasi-geodesic by Lemma~\ref{concate}. Similarly, as $ p_j $ is the closest point on $ [y_{i-1}, p_i] \cdot [p_i, p_{i+1}] \cdots [p_{j-1}, p_j] $ to $ z_j $, it follows from Lemma~\ref{concate} that $ \sigma $ is a $ (9\mu, 9\mu) $-quasi-geodesic.

We then have

\[
d(y_{i-1}, z_j) \ge \frac{\operatorname{Length}(\sigma)}{9\mu} - 9\mu \ge \frac{d(p_j, z_j)}{9\mu} - 9\mu \ge \frac{\kappa_j}{9\mu} - 9\mu = \frac{r (1 + \rho_0)^j}{9\mu} - 9\mu.
\]

Since $ 81\mu^2 < \frac{r}{2} < \frac{r}{2}(1 + \rho_0)^j $, it follows that

\[
\frac{r (1 + \rho_0)^j}{9\mu} - 9\mu \ge \frac{r (1 + \rho_0)^j}{18\mu}.
\]

We thus can control the upper bound of the ratio:
\[
\frac{ \operatorname{Length} \bigl ( L_{r, k}|_{[y_{i-1}, z_j]} \bigr )}{ d(y_{i-1}, z_{j})} \le \frac{72 \mu }{\rho_0} (1+ \rho_0)
\]

A similar argument shows that there is an uniform constant $\Delta = \Delta(\mu, \rho_0)$ such that  for any points $x, y$ in $L_{r, k}$, we have
\[
\frac{\operatorname{Length}\bigl (L_{r,k}|_{[x,y]} \bigr )}{ d(x,y)} \le \Delta
\]
In other words, $L_{r, k}$ is a $(\Delta, \Delta)$--quasi-geodesic for every $k < n$.

\subsection*{Growths of forward spiral paths of Type I}
The collection of forward spiral paths $ L_{r,k}$ with $k <n$ has the following property: 
\begin{align*}
    \operatorname{Length}(L_{r,k+1}) - \operatorname{Length}(L_{r,k}) &\leq d(y_{k-1}, p_k) + d(p_k, z_k) + d(z_k, y_k) \\
    &\leq \kappa_{k-1} + \kappa_k + \delta_k \\
    &< 2r(1+\rho_0)^k + \delta_k.
\end{align*}

\textit{Claim:} There exists a constant $ \rho = \rho(\rho_0) $, which tends to 0 as $ \rho_0 \to 0 $, such that:

\[
2r(1+\rho_0)^k < \rho \operatorname{Length}(L_{r,k}) 
\]

for sufficiently large $ k $.

Indeed, we have:
\begin{align*}
 \operatorname{Length}(L_{r,k}) &= r + d(\zeta(r), z_1) + d(z_1, y_1) + \cdots + d(z_{k-1}, y_{k-1}) \\
 &\geq r + \sum_{i=1}^{k-1} \kappa_i \\
 &= r + r \sum_{i=1}^{k-1} (1+\rho_0)^i   \\
 & \geq r + r \frac{(1+\rho_0)}{\rho_0} \left( (1+\rho_0)^{k-1} - 1 \right) \\
 &= r \left( \frac{(1+\rho_0)^k}{\rho_0} - \frac{1}{\rho_0} \right).
\end{align*}

Thus,

\[
\frac{2r (1+\rho_0)^k}{\operatorname{Length}(L_{r,k})} \leq \frac{2\rho_0 (1+\rho_0)^k}{(1+\rho_0)^k - 1} \to 2\rho_0
\]

as $ k \to \infty $. Therefore, there exists $ n_0 \in \mathbb{N} $ such that for all $ k \geq n_0 $, we have:

\[
\frac{2r (1+\rho_0)^k}{\operatorname{Length}(L_{r,k})} < 3\rho_0 =: \rho.
\]

This confirms the claim. We then conclude:

\[
\operatorname{Length}(L_{r,k+1}) < (1+\rho) \operatorname{Length}(L_{r,k}) + \delta_k
\]

for sufficiently large $ k $.

We summarize the above discussion in the next proposition.

\begin{prop}[N-Qing-Margolis]
\label{prop:smallslope}
 Let $\alpha$ be a $\qq$--ray of Type II of sub-exponential excursion. Let $\rho >0$ be an arbitrary sufficiently small constant. Then there exists a sufficiently large $k_0 = k_0(\rho)$ and a constant $C >0$ such that for every $r > C$ and for every $k \ge k_0$, the collection of forward spiral path of Type~I $\{L_{r,i} \}_{i < k}$ constructed as above satisfies the following properties:
 \begin{enumerate}
     \item It is a $(L, A)$--quasi-geodesic with some constants $L, A$ are independent of $r, k$.
     \item 
     \[
\operatorname{Length}(L_{r,k+1}) < (1+\rho) \operatorname{Length}(L_{r,k}) + \delta_k
\]
 \end{enumerate}
\end{prop}

\subsection{Forward spiral paths: Type II}
\label{subsection:forward-spiral-II}

In this section, we define and analyze forward spiral paths of Type II, which differ from Type I in their geometric construction and quasi-geodesic properties. These paths exhibit controlled sub-exponential growth, ensuring their use in the study of QR boundaries.

Given a sufficiently small constant $\rho_0 \in (0, 1/16)$, $r >0$. Let $t_1, t_2, \ldots t_k$ be positive numbers so that 
\[
\begin{cases}
t_k - t_{k-1} &\ge r (1 + \rho_0)^k \\
t_i - t_{i-1} &< r (1 +\rho_0)^i \quad
\forall 1 \le i \le k-1
\end{cases}
\]
We define 
\[
\kappa_{i} \coloneqq r \rho_0 (1+ \rho_0)^i
 \quad \text{for each} \quad 1  \le i \le k  \]

We have
\[ \kappa_k = r \rho_{0} (1 + \rho_0)^k < \rho_{0} (t_k - t_{k-1}) 
\]
and hence it implies that  $$ \frac{\kappa_k} { \rho_0} < t_k - t_{k-1}$$ 

\begin{figure}[htb]
\centering 
 \def\svgwidth{0.5\textwidth}
 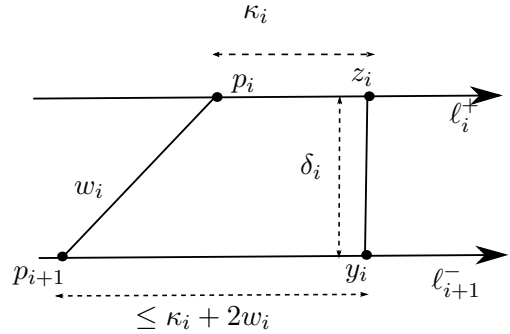
\caption{The figure illustrates how we choose geodesic rays $\ell_{i}^{+}$ and $\ell_{i+1}^{+}$ on the strip which is slightly different from Figure~\ref{choiceofelli} as the constant $\kappa_i$ is chosen to be smaller than $w_i$. }\label{choiceofellimodify}
\end{figure}

\begin{enumerate}
    \item (Choosing geodesic rays) For each $i$, we define two geodesic rays:
    \begin{itemize}
        \item $\ell_{i}^{+} \subset f_{i}^{+}$ based at $p_i$.

        \item $\ell_{i+1}^{-} \subset f_{i+1}^{-}$ based at $p_{i+1}$.
    \end{itemize}
   We ensure that the projection point of $p_{i}$ into $f_{i+1}^{-}$ will belong to $\ell_{i+1}^{-}$ (see Figure~\ref{choiceofellimodify}). 

\item (Defining the forward spiral segments) On $\ell_{i}^{+}$ choose $z_i$  such that
$$d(p_i, z_i) = \kappa_i$$
Let $y_i$ be the projection of $z_i$ on $\ell_{i+1}^{-}$.  We have
\begin{itemize}
    \item $d(z_i, y_i) = \delta_i$
    \item $d(y_i, p_{i+1}) \le \kappa_i + 2 w_i$.
\end{itemize}

    \item (Concatenating the path) Define the \textit{forward spiral path of Type II}  $J_{r, k}$  as the concatenation:
\[
J_{r, k} \coloneqq \zeta|_{[0, r]} \cdot [\zeta(r), z_1] \cdot [z_1, y_1] \cdot [y_1, z_2] \cdots [y_{k-1}, z_{k-1}]
\]
where
$ \zeta $ is the given main flat ray.
\end{enumerate}

Using similar arguments as forward spiral paths of Type I, we can verify that $J_{r,k}$ is a $(L, A)$--quasi-geodesic for some constantS $L = L (\rho_0, \qq)$, $A = A (\rho_0, \qq)$. 

\begin{lem}
    There are constants $L = L (\rho_0, \qq)$, $A = A (\rho_0, \qq)$ such that every spiral path of Type II $J_{r,l}$ is a $(L,A)$--quasi-geodesic.
\end{lem}

In this section, we have constructed backward spiral paths and forward spiral paths in templates and established that they are quasi-geodesics. This provides a crucial tool for proving that QR boundaries exist in Croke-Kleiner admissible groups. In the next section, we apply these results to show that the QR boundary of these groups is well-defined and contains important structural information.

\section{QR boundary of  Croke-Kleiner admissible groups}
\label{sec:QR-admissible-groups}
$\CAT(0)$ Croke-Kleiner admissible groups were  introduced by Croke--Kleiner in~\cite{CK02}. They are a particular class of graph
of groups that includes fundamental groups of 3-dimensional graph manifolds. The QR-boundary of a specific case of $\CAT(0)$ admissible group is computed in~\cite{QR24}. In this section we follow the arguments in~\cite[Section 11]{QR24} closely but  adapt and expand them to suit all $\CAT(0)$ Croke-Kleiner admissible groups. We remark here that Croke-Kleiner admissible groups are not relatively hyperbolic groups (see~\cite[Lemma~4.7]{MN24}).

\begin{defn}
	A \emph{graph of groups} $\mathcal{G} = (\Gamma, \{G_{v}\}, \{G_{e} \}, \{\tau_{e} \})$ consists of the following data:
	\begin{enumerate}
		\item a graph $\Gamma$, called the \emph{underlying graph},
		\item a group $G_v$ for each vertex $v \in V \Gamma$, called a \emph{vertex group},
		\item a subgroup $G_e\leq G_{e_-}$ for each edge $e \in E \Gamma$, called an \emph{edge group},
		\item an isomorphism $\tau_{e} \colon {G}_{e} \to {G}_{\bar e}$  for each $e\in E\Gamma$ such that $\tau^{-1}_e=\tau_{\bar e}$, called an \emph{edge map}.
	\end{enumerate}
\end{defn}
The \emph{fundamental group} $\pi_{1}(\cG)$ of a graph of groups $\cG$ is as defined in~\cite{SW79}.

\begin{defn}\label{defn:admissible}
A graph of groups $\mathcal{G}$ is \emph{admissible} if
\begin{enumerate}
    \item $\mathcal{G}$ is a finite graph with at least one edge.
    \item Each vertex group ${ G}_v$ has center $Z({ G}_v) \cong \Z$, ${Q}_v \coloneqq { G}_{v} / Z({ G}_v)$ is a non-elementary hyperbolic group, and every incident edge subgroup ${ G}_{e}$ is isomorphic to $\Z^2$. 
    \item Let $e_1$ and $e_2$ be distinct directed edges entering a vertex $v$, and for $i = 1,2$, let $K_i \subset { G}_v$ be the image of the edge homomorphism ${G}_{e_i} \to {G}_v$. Then for every $g \in { G}_v$, $gK_{1}g^{-1}$ is not commensurable with $K_2$, and for every $g \in  G_v \setminus K_i$, $gK_i g^{-1}$ is not commensurable with $K_i$.
    \item For every edge group ${ G}_e$, if $\alpha_i \colon { G}_{e} \to { G}_{v_i}$ is the edge monomorphism, then the subgroup generated by $\alpha_{1}^{-1}(Z({ G}_{v_1}))$ and $\alpha_{2}^{-1}(Z({ G}_{v_1}))$ has finite index in ${ G}_e$.
\end{enumerate}
\end{defn}

\begin{defn}
A group $G$ is \emph{admissible} if it is the fundamental group of an admissible graph of groups. We say that a Croke-Kleiner  admissible group $G$ is a \emph{$\CAT(0)$ Croke-Kleiner  admissible group} if there is  a complete proper $\CAT(0)$
space $X$ such that $G$ acts on $X$ isometrically,  properly discontinuously and cocompactly. Such an action $G \curvearrowright X$ is called a \emph{CKA action} and the space $X$ is called a \emph{$\CAT(0)$ admissible space} of $G$.
\end{defn}

Below are some examples of $\CAT(0)$ Croke-Kleiner admissible groups.
\begin{exmp} \label{exmp:CKA} \hfill
   \begin{enumerate}
       \item (Tori complexes) \label{toruscomplexes}
        Let $n \ge 3$ be an integer. Let $T_1, T_2, \ldots, T_n$ be a family of flat two-dimensional tori. For each $i$, we choose a pair of simple closed geodesics $a_i$ and $b_i$ such that $\operatorname{length} (b_i) = \operatorname{length} (a_{i+1})$, identifying $b_i$ and $a_{i+1}$ and denote the resulting space by $X$.
		      The space $X$ is a graph of spaces with $n-1$ vertex spaces $V_{i} \coloneqq T_{i} \cup T_{i+1} / \{ b_i =  a_{i+1} \}$ (with $i \in \{1, \ldots, n-1\}$) and $n-2$ edge spaces $E_{i} \coloneqq V_{i} \cap V_{i+1}$.

		      The fundamental group $G = \pi_{1}(X)$ has a graph of groups structure where each vertex group is the fundamental group of the product of a figure eight and $S^1$. Vertex groups are isomorphic to  $F_{2} \times \Z$ and edge groups  are isomorphic to $\pi_1(E_i) \cong \Z^2$. The generators $[a_i], [b_i]$ of the edge group $\pi_1(E_i)$ each map to a generator of either a  $\Z$ or  $F_2$ factor of $F_{2} \times \Z$. It is clear that with this graph of groups structure, $\pi_1(X)$ is a Croke-Kleiner  admissible group.

       \item (Graph manifolds) \label{exmp:graphmanifolds} Let $M$ be a non-geometric graph manifold that admits a nonpositively curved metric. Lift this metric to the universal cover $\tilde{M}$ of $M$, and we denote this metric by $d$. Then the action $\pi_1(M) \curvearrowright (\tilde{M}, d)$ is a CKA action.

       \item \label{exmp:contructionflip}  One may build $\CAT(0)$ Croke-Kleiner admissible groups algebraically from any finite number of  hyperbolic $\CAT(0)$ groups. The following example is for $n =2$ but the same principle works for any $n \ge 2$. Let $H_1$ and $H_2$ be two torsion-free hyperbolic groups that act geometrically on $\CAT(0)$ spaces $X_1$ and $X_2$ respectively. Then $G_i = H_i \times \langle t_i \rangle$ (with $i =1,2$)  acts geometrically on the $\CAT(0)$ space $Y_ i = X_{i} \times \mathbb R$. Any primitive hyperbolic element $h_i$ in $H_i$ gives rise to a totally geodesic torus $T_i$ in the quotient space $Y_{i}/G_i$ with basis $([h_i], [t_i])$. We re-scale $Y_ i$ so that the translation length of $h_i$ is equal to that of $t_i$ for each $i$.  Let $f \colon T_1 \to T_2$ be a \textit{flip} isometry respecting these lengths, that is, an orientation-reversing isometry mapping $[h_1]$ to $[t_2]$ and $[t_1]$ to $[h_2]$.   Let $M$ be the space obtained by gluing $Y_1$ to $Y_2$ by the isometry $f$. There is a metric on the space $M$ that gives rise to a locally $\CAT(0)$ space (see e.g.\ \cite[Proposition II.11.6]{BH99}). By the Cartan--Hadamard Theorem, the universal cover $\widetilde M$ with the induced length metric from $M$ is a $\CAT(0)$ space. Let $G$ be the fundamental group of $M$. Then the action $G \curvearrowright \widetilde M$ is geometric, and $G$ is an example of a Croke--Kleiner admissible group.  
   \end{enumerate} 
\end{exmp}

\subsection{Vertex and edge spaces in $\CAT(0)$ admissible spaces}\label{sec:blocks}
Let $G$ be a Croke-Kleiner  admissible group that acts properly discontinuously, cocompactly, and by isometries on a complete proper $\CAT(0)$ space $X$. 
Let $G \curvearrowright T$ be the action of $G$ on the associated Bass--Serre tree $T$ of the graph of group $\mathcal{G}$ (we refer the reader to~\cite[Section~2.5]{CK02} for a brief discussion).

Let $V(T)$ and $E(T)$ be the vertex and edge sets of $T$.  
For each ${\sigma} \in V(T) \cup E(T)$, we let $G_{ \sigma} \le G$ be the stabilizer of $\sigma$. 
We review facts from~\cite[Section~3.2]{CK02} that will be used thoroughly in this paper, and refer the reader to~\cite{CK02} for further explanation.
From the given actions $G \curvearrowright X$ and $G \curvearrowright T$ we have
\begin{enumerate}
    \item for every vertex ${v}\in V(T)$, the set $X_{v}\coloneqq\cap_{g\in Z(G_{v})} \operatorname{Minset}(g)$ splits as metric product \[X_{ v} = H_{  v} \times \mathbb R\] where $Z(G_{v})$ acts by translation on the $\mathbb{R}$--factor and  the quotient $Q_{v}\coloneqq G_{ v}/Z(G_{ v})$ acts geometrically on the $\CAT(0)$ space $H_{v}$. 
    
    \item for every edge $ e \in E(T)$, the minimal set $X_{e}\coloneqq\cap_{g\in G_{e}} \operatorname{Minset}(g)$ splits as \[X_e=\overline{X_{e}} \times \mathbb R^2\subset X_{v},\] where $\overline{X_{e}}$ is a compact $\CAT(0)$ space and $G_{e}=\mathbb Z^2$ acts co-compactly on the Euclidean plane $\mathbb R^2$.

\end{enumerate}
\begin{defn}
    For every vertex $v \in V(T)$ and edge $e \in E(T)$, the spaces $X_v$ and $X_e$ are called the \textit{vertex space} and \textit{edge space} of $X$ respectively.
\end{defn}

\begin{rem}
    For each vertex space $X_{ v}$, since the quotient $Q_{ v}\coloneqq G_{ v}/Z(G_{ v})$ is a non-elementary hyperbolic group and acts geometrically on $H_{ v}$, it follows that $H_{ v}$ is a hyperbolic space.
\end{rem}

In the sequel, it will be   useful to make the following specific choices.
\begin{defn}\label{defn:indexfunction}
There exists a $G$-equivariant coarse $L$--Lipschitz map  $\gothic i \colon X \to T^0$  such that  $x \in X_{\gothic i (x)}$ for all $x \in X$. The map $\gothic i$ is called an \emph{index map}. We refer the reader to Section~3.3 in~\cite{CK02} for the existence of such a map $\gothic i$.

\end{defn}


\subsection{Admissible strips and admissible planes in $\CAT(0)$ admissible spaces}\label{sec:strips}\cite[Section~4.2]{CK02}
We note that the assignments $ v \to X_{ v}$ and $ e \to X_{ e}$ are $G$-equivariant  in the sense that $g X_{ v}  = X_{g  v}$ and $g X_{ e} = X_{g  e}$ for any $g \in G$.

\begin{defn}[admissible planes and admissible strips]
 We first choose, in a $G$-equivariant way, a plane  $F^{*}_{ e} \subset X_{ e}$ which we will call the \emph{admissible plane} for each edge $ e \in E(T)$.
 
 For every pair of adjacent edges ${e}$, ${e}'$, we choose, again equivariantly, a minimal geodesic from $F^{*}_{{e}}$ to $F^{*}_{{e}'}$; by the convexity of $X_{ v} = H_{ v} \times \R$ where $ v\coloneqq {e} \cap {e}'$, this geodesic determines an \textit{admissible strip} in the $\CAT(0)$ admissible space $X$: 
 \[
 \mathcal{S}^{*}_{{e} {e}'} \coloneqq h_{{e} {e}'} \times \R\subset X_v
 \]
 (possibly of width zero) for some geodesic segment $h_{{e} {e}'} \subset H_v$. 
\end{defn}

\begin{rem} \hfill
\begin{enumerate}
   
    \item Note that lines $\mathcal{S}^{*}_{{e} {e}'} \cap F^{*}_{{e}}$ and $\mathcal{S}^{*}_{{e} {e}'} \cap F^{*}_{{e}'}$ are axes of $Z(G_{{v}})$. Hence if ${e}, {e}', {e}'' \in E(T)$ are three consecutive edges,  then the angle between the geodesics $\mathcal{S}^{*}_{{e} {e}'} \cap F^{*}_{{e}'}$ and $\mathcal{S}^{*}_{{e}' {e}''} \cap F^{*}_{ e'}$ is bounded away from zero.

    \item Suppose $e$ is adjacent to incident to vertices $v_1,v_2\in V(T)$. If $\langle f_1\rangle=Z(G_{v_1}), \langle f_2\rangle=Z(G_{v_2})$, then $\langle f_1, f_2\rangle$ generates a finite index subgroup of $G_e$. Thus if $e_1,e, e_2$ are three consecutive edges,  the intersection of the two admissible strips $\mathcal{S}^{*}_{{e}_1  e}$ and $\mathcal{S}^{*}_{{e}_2 {e}}$ is a point. 
    Indeed, we have 
    \[
    \mathcal{S}^{*}_{ {e}_1  e} \cap \mathcal{S}^{*}_{{e}_2 {e}} = (\mathcal{S}^{*}_{{e}_1 {e}} \cap F^{*}_{ e}) \cap (\mathcal{S}^{*}_{{e}_2 {e}} \cap F^{*}_{ e}).
    \]
    As two lines $\mathcal{S}^{*}_{{e}_1 {e}} \cap F^{*}_{ e}$ and $\mathcal{S}^{*}_{{e}_2  e} \cap F^{*}_{ e}$ in the wall $F^{*}_{ e}$ are axes of $\langle f_{{v}_1} \rangle = Z(G_{{v}_1})$, $\langle f_{{v}_1} \rangle = Z(G_{{v}_2})$ respectively and $\langle f_1, f_2\rangle$ generates a finite index subgroup of $G_{ e}$, it follows that these two lines are non-parallel, and hence their intersection must be a single point.
    \end{enumerate}
\end{rem}

Recall that each $X_v$ decomposes as a metric product of a hyperbolic Hadamard space $H_v$ with the real line $\mathbb R$ such that $H_v$ admits a geometric action of $Q_v$.  Recall  that we choose a $G$--equivariant family of Euclidean planes $\{F^{*}_e: F^{*}_e\subset X_e\}_{e \in E(T)}$. 

\begin{defn} 
    The space $X$ is called a \textit{flip admissible space} if for each edge $e:=[v,w]\in E(T)$,  the   boundary line  $\ell := H_v \cap F^{*}_e$ is parallel  to the $\R$--line in $X_w = H_w \times \R$. We also call the group $G$ a \textit{flip admissible group}.
\end{defn}

\begin{exmp}  
Examples (\ref{toruscomplexes}) and (\ref{exmp:contructionflip}) in Example~\ref{exmp:CKA} are instances of flip-admissible groups. The \textit{flip graph manifolds} introduced by Kapovich–Leeb \cite{KL98} are also typical examples of flip-admissible spaces. A \emph{flip manifold} is a graph manifold constructed as follows: Take a finite collection of products of $ S^1 $ with compact orientable hyperbolic surfaces and glue them along boundary tori using maps that interchange the base and fiber directions. Kapovich–Leeb proved that for any graph manifold $ M $, there exists a flip graph manifold $ N $ whose fundamental group is quasi-isometric to that of $ M $.  
\end{exmp}

\subsection{Embedded templates into admissible spaces}
Let $X$ be a CAT(0) admissible space.
We are going to recall a template associated with a geodesic in the Bass-Serre tree in \cite[Section 4.2]{CK02}  as the following.

\begin{defn}
\label{defn:standardtemplate}
Let $\gamma$ be a geodesic segment or ray in the Bass-Serre tree $T$. We may write $\gamma = e_1 \cdot e_2 \cdots e_k$ (or $\gamma = e_1 \cdot e_2 \cdots e_k \cdots$ in case $\gamma$ is a geodesic ray). 

We begin with a collection of walls $F_e$ and an isometry $\phi_e \colon F_e \to F^{*}_e$ for each edge $e \subset \gamma$. For every pair $e_i$, $e_{i+1}$ of adjacent edges of $\gamma$, we let $S_{i (i+1)}$ be a strip which is isometric to $\mathcal{S}^{*}_{e_i e_{i+1}}$ if the width of  $\mathcal{S}^{*}_{e_i e_{i+1}}$ is at least $1$, and isometric to $[0,1] \times \R$ otherwise; we let $\phi_{e_i e_{i+1}} \colon S_{i(i+1)} \to \mathcal{S}^{*}_{e_i e_{i+1}}$ be an affine map which respects product structure ($\phi_{e_i e_{i+1}}$ is an isometry if the width of $\mathcal{S}^{*}_{e_i e_{i+1}}$ is greater than or equal to $1$ and compresses the interval otherwise).

\textit{The standard template $\mathcal{T}$ associated with $\gamma \subset T$} is the template obtained by 
 gluing the strips $S_{i (i+1)}$ and walls $F_e$ so that the maps $\phi_{e_i}$ and $\phi_{e_i e_{i+1}}$ descend to continuous maps on the quotient, we denote the map from $\mathcal{T}_{\gamma} \to X$ by $\phi_{\gamma}$.
\end{defn}

The following lemma is cited from \cite[Lemma~4.5]{CK02} and \cite[Proposition~4.6]{CK02} which basically say each template associated with a geodesic segment/ ray in the Bass-Serre tree $T$ is quasi-isometrically embedded in $X$ with uniform quasi-isometric constants.

\begin{lem}
\label{lem:CK02}
Let $G \curvearrowright X$ be a CKA action. Then
\begin{enumerate}
    \item There exists $\beta = \beta(X) >0$ such that the following holds. For any geodesic segment $\gamma \in T$, the angle function $\alpha_{\gamma} \colon Wall_{\mathcal{T}_{\gamma}} \to (0, \pi)$ satisfies $0 < \beta  \le \alpha_{\gamma} \le \pi - \beta < \pi$.
    \item There are universal constants $L, A >0$ such that the following holds. Let $\gamma$ be a geodesic segment in $T$, and let $\phi_{\gamma} \colon \mathcal{T}_{\gamma} \to X$ be the map given by Definition~\ref{defn:standardtemplate}. 
    Then $\phi_{\gamma}$ is a $(L,A)$--quasi-isometric embedding. Moreover, for any $x, y \in [\cup_{e \subset \gamma} X_{e}] \cup [\cup_{ee' \subset \gamma} \mathcal{S}^{*}_{ee'}]$, there exists a continuous map $\alpha \colon [x,y] \to \mathcal{T}_{\gamma}$ such that $d(\phi_{\gamma} \circ \alpha , id|_{[x,y]}) \le L$.
\end{enumerate}
\end{lem}

\subsection{Main flat rays}
\label{sec:mainflatray}
This section assumes that $X$ is a flip admissible space.
Let $\gothic i \colon X \to T$ be the index map given by Definition~\ref{defn:indexfunction}, and fix a admissible plane $F^{*}$ in $X$. We also assume that the basepoint $\gothic o \in F^{*}$ and $F^{*} \subset X_{v_0}$ where $v_0 \coloneqq \gothic{i} (\go)$.
Recall that $X_{v_0}$ splits as a metric product $H_{v_0} \times \R$. In the rest of this paper, we fix a geodesic ray $\zeta^{*}$ based at $\go$ that follows the line $\R$ in the $\R$ factor of $X_{v_0}$, and call it the \textit{admissible main flat ray}.

We remark that the choice of $\zeta^{*}$ is arbitrary since any quasi-geodesic ray in $X_{v_0}$ is in the same equivalent class as $\zeta^{*}$ by Proposition~\ref{prop:product}.

We first show that every $\qq$--ray $\alpha^{*}$  can be quasi-redirected to the admissible main flat ray $\zeta^{*}$ at every radius $r >0$, via a quasi-geodesic $\gamma^{*}_r$ with uniform quasi-geodesic constants; see Proposition~\ref{Prop:zeta-is-top}.

\begin{prop}[Quasi-redirecting to the main flat ray]
\label{Prop:zeta-is-top} 
Let $\alpha^{*}$ be a $\qq$--ray in the flip admissible space $X$. Then $\alpha^{*}$ can be quasi-redirected to the admissible main flat ray $\zeta^{*}$  at every radius $r>0$ via a quasi-geodesic $\gamma^{*}_r$ with uniform quasi-geodesic constants, In particular, we have $\alpha^{*} \preceq \zeta^{*}$.
 \end{prop}
 \begin{proof}
   If $\alpha^{*}$ does not intersect any admissible plane, then $\alpha^{*}$ necessarily lies in a neighborhood of same vertex space as the basepoint $\go$.  By Proposition~\ref{prop:product}, $\alpha^{*}$ and $\zeta^{*}$ redirect to each other. Otherwise, $\alpha^{*}$ intersects a non-empty (finite or infinite)  collection of admissible planes. 

   Given $q = 81 \sqrt{2}$, $Q =1$, $\delta \in (0,1]$, and $\rho > q/\delta + Q$.

   Choose a sequence $\{t^{*}_n > 0\}$ so that $t^{*}_n \to \infty$. For each $n$, choose a admissible plane, denoted by $F^{*}_{e_n}$, that is that is sufficiently far from $\alpha^{*}(t^{*}_n)$ so that if $w^{*}_n$ is a point in $\alpha^{*}|_{[0, t^{*}_n]}$  such that
         \[
         d(w^{*}_n, F^{*}_{e_{n}}) = \inf \left\{ d(x, F^{*}_{e_{n}}) \;\middle|\; x \in \alpha^{*}|_{[0, t^{*}_n]} \right\}
         \] then $d(\go , w^{*}_n) \to \infty$ when $n \to \infty$ and  $r_n : = d(w^{*}_n, F^{*}_{e_{n}}) > \rho \, \sup_{0\leq s\leq t^{*}_n} d(\go, \alpha^{*}(s))$ 

         Let $x^{*}_{n}$ be a point in $F^{*}_{n}$ that realizes $d(w^{*}_n, F^{*}_{e_{n}})$, and denote $d(x^{*}_{n}, w^{*}_n)$ by $r_n$.  By Lemma~\ref{concate}, we have that the concatenation $[\go, w^{*}_n]_{\alpha^{*}} \cup [w^{*}_n, x^{*}_{n}]$ is a $(3q, Q)$--quasi-geodesic. Also, if $\ell$ is a $(q, Q)$--quasi-geodesic starting at $x^{*}_{n}$ and contained in $F^{*}_{e_{n}}$, then the concatenation $[w^{*}_n, x^{*}_{n}] \cup \ell$ is also a $(3q, Q)$--quasi-geodesic by Lemma~\ref{concate}.

         Let $\gamma \subset T$ be the geodesic segment in the tree $T$ starting at $v_0$ and ending at $(e_n)_{+}$ (where $e_n$ is the last edge of $\gamma$). Define $\mathcal{T}$ as the standard template in $\mathbb{R}^3$ associated with $\gamma$, as per Definition~\ref{defn:standardtemplate}. Since $X$ is a flip-admissible space, this template is right-angled. We write  $$\mathcal{T} : = F_0 \cup S_{01} \cup F_1 \cup S_{12} \cup \ldots \cup F_n$$ By Lemma~\ref{lem:CK02}, there exists an $(A_1,A_2)$--quasi-isometric embedding $\Phi \colon \mathcal{T} \to X$, where $A_1$ and $A_2$ are uniform constants depending only on the geometry of $X$. Choose $x_n \in F_n$ such that $\Phi(x_n) = x^{*}_n$.

Given that $q = 81 \sqrt{2}$, $Q =1$, $\delta \in (0,1]$, and $\rho > q/\delta + Q$, the backward spiral path $$Z_{q, Q, \delta, \rho, x_n} = Z_{n} \cdot Z_{n-1} \cdots Z_1$$ as constructed in Section~\ref{subsec:backwardspiral}, is a $(L,C)$--quasi-geodesic by Corollary~\ref{cor:backwardlogisquasigeodesic}. Moreover, the part $v_n$ in $Z_n$, introduced in part~(1) of the construction in Section~\ref{subsec:backwardspiral}, can be chosen arbitrarily large. Hence, we can select $v_n$ so that the ratio $\operatorname{Length}(v_n)/ d(w^{*}_n, x^{*}_n)$ is sufficiently large.

          We define $$\mathcal{Z}^{*} : = \Phi(\mathcal{Z}_{q, Q, \delta, \rho, x_n}) = Z^{*}_n \cdot Z^{*}_{n-1} \cdots Z^{*}_1$$ where $Z^{*}_i = \Phi(Z_i)$.
          
It is straightforward to verify that $[\go, w^{*}_n]_{\alpha^{*}}$, $[w^{*}_n, x^{*}_n]$, $Z^{*}_n$, ..., $Z^{*}_1$ satisfy conditions (\ref{item1}), (\ref{item2}), (\ref{item3}), (\ref{item4}), (\ref{item5}) in Proposition~\ref{prop:constructQg} for some $q^{*} \ge 1$, $Q^{*} \ge 0$, $\delta^{*} \in (0,1]$ and $\rho^{*}$. Consequently, the path $$\gamma^{*}  : = [\go, w^{*}_n]_{\alpha^{*}} \cdot [w^{*}_n, x^{*}_n] \cdot Z^{*}$$ is a $(L^{*}, C^{*})$--quasi-geodesic where $L^{*}$ and $C^{*}$ are constants given by Proposition~\ref{prop:constructQg}. 

Thus, we have shown that $\alpha^{*}$ can be quasi-redirected to $\zeta^{*}$ at $w^{*}_n$ via $\gamma^{*}$, which is an $(L^{*}, C^{*})$--quasi-geodesic. Since $d(\go , w^{*}_n) \to \infty$, Lemma~\ref{lem:infinitepoints} implies that $\alpha^{*}$ can be $(q', Q')$--quasi-redirected to $\zeta^{*}$.
\end{proof}

\subsection{Type I and Type II quasi-geodesics}
\begin{defn}\label{defn:typesIandII} 
Let ${\alpha^{*}}$ be an arbitrary ${\qq}$--ray in the CKA space $X$ emanating from ${\gothic o}$.   Recall that ${\qq}$--rays are always assumed to be continuous. 
\begin{itemize}
    \item Let $u_1>0$ be the supremum of times $t$ such that $\alpha^{*}(t)$ lies in the vertex space $X_{v_0}$. If $u_1 = \infty$ we stop here.
    \item If $u_1$ is finite, we then let $X_{v} \neq X_{v_0}$ be the vertex space $\alpha^{*}$ enters immediately after it exists $X_{v_0}$ and define $u_2$ to be the supremum of times $t$ such that $\alpha^{*}(t)$ lies in $X_{v}$.
    \item Repeat this process to define a sequence $u_1 <u_2 < \ldots$ as long as $u_i$ is finite.
\end{itemize}
We classify $\alpha^{*}$ into two types

\begin{enumerate}
    \item Type I: If there exists an index $i$ such that $u_i = \infty$, then $\alpha^{*}$ remains in a finite set of vertex spaces. 
    \item Type II: If $u_i$ is finite for all $i$, then  the radii of  $v_i$ (i.e, $d_{T}(v_0, v_i)$) in $T$ tends to infinity monotonically. Since $T$ is a tree, there is exactly one geodesic ray whose vertex set is contained in $\gothic i ({\alpha^{*}})$. Denote this geodesic ray $\gamma^{*}$.  Relabel again such that ${\gamma^{*}}$ traverses vertices ${v}_0$, ${v}_1$, ${v}_2, \dots$ etc. In this case, we say the ${\qq}$--ray ${\alpha^{*}}$ is of \emph{Type~II}.
\end{enumerate} 

\end{defn}

\begin{figure}[htb]
\centering 
 \def\svgwidth{0.8\textwidth}
\begingroup%
  \makeatletter%
  \providecommand\color[2][]{%
    \errmessage{(Inkscape) Color is used for the text in Inkscape, but the package 'color.sty' is not loaded}%
    \renewcommand\color[2][]{}%
  }%
  \providecommand\transparent[1]{%
    \errmessage{(Inkscape) Transparency is used (non-zero) for the text in Inkscape, but the package 'transparent.sty' is not loaded}%
    \renewcommand\transparent[1]{}%
  }%
  \providecommand\rotatebox[2]{#2}%
  \newcommand*\fsize{\dimexpr\f@size pt\relax}%
  \newcommand*\lineheight[1]{\fontsize{\fsize}{#1\fsize}\selectfont}%
  \ifx\svgwidth\undefined%
    \setlength{\unitlength}{488.09098371bp}%
    \ifx\svgscale\undefined%
      \relax%
    \else%
      \setlength{\unitlength}{\unitlength * \real{\svgscale}}%
    \fi%
  \else%
    \setlength{\unitlength}{\svgwidth}%
  \fi%
  \global\let\svgwidth\undefined%
  \global\let\svgscale\undefined%
  \makeatother%
  \begin{picture}(1,0.25114053)%
    \lineheight{1}%
    \setlength\tabcolsep{0pt}%
    \put(0,0){\includegraphics[width=\unitlength,page=1]{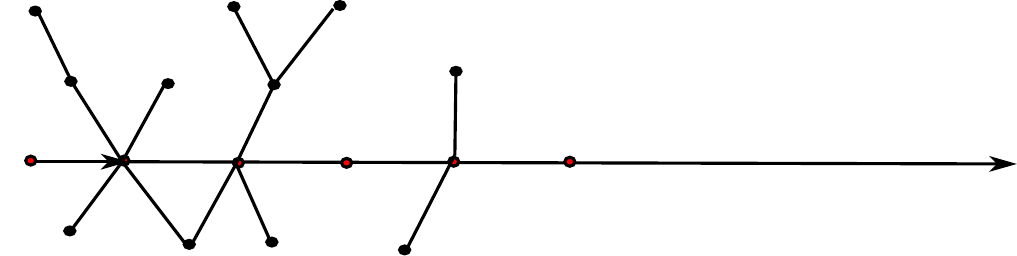}}%
    \put(-0.00048758,0.06026481){\color[rgb]{1,0,0}\makebox(0,0)[lt]{\lineheight{1.25}\smash{\begin{tabular}[t]{l}$v_0$\end{tabular}}}}%
    \put(0.14878201,0.10965543){\color[rgb]{1,0,0}\makebox(0,0)[lt]{\lineheight{1.25}\smash{\begin{tabular}[t]{l}$v_1$\end{tabular}}}}%
    \put(0.26073421,0.05806976){\color[rgb]{1,0,0}\makebox(0,0)[lt]{\lineheight{1.25}\smash{\begin{tabular}[t]{l}$v_2$\end{tabular}}}}%
    \put(0.3309787,0.10746038){\color[rgb]{1,0,0}\makebox(0,0)[lt]{\lineheight{1.25}\smash{\begin{tabular}[t]{l}$v_3$\end{tabular}}}}%
  \end{picture}%
\endgroup%

\caption{The figure illustrates a portion of vertices $\gothic{i}(\alpha^{*})$ visits. With respect to $\gothic{i}(\alpha^{*})$, there is the unique geodesic ray $\gamma_{\alpha^{*}} \coloneqq [v_0, v_1] \cdot [v_1, v_2] \cdot [v_2, v_3] \cdots$ associated to $\alpha^{*}$.}\label{itinerary}
\end{figure}
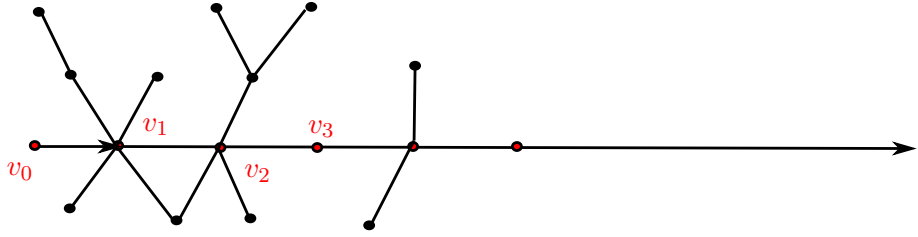

In the case $\alpha^{*}$ is of Type II, since $\gamma^{*}$ as above  is uniquely determined by $\alpha^{*}$, we denote it by $\gamma_{\alpha^{*}}$. We call $\gamma_{\alpha^{*}}$ and the associated ordered, infinite sequence of vertices $v_0, v_1, v_2,\dots$, the \textit{simplified itinerary} associated to $\alpha^{*}$. 

We define $e_{i} \coloneqq [v_{i-1}, v_i]$ and let $v_0 \coloneqq \gothic i(\go)$. Let $h_{e_0e_1}$ be the geodesic in $H_{v_0}$ realizing the shortest geodesic between $\pi_{v_0}(\go)$ and $\pi_{v_0}(F^{*}_{e_1})$, where $\pi_{v_0}:X_{v_0}=H_{v_0} \times \R\to H_{v_0}$ is the projection. Let  $\mathcal{S}^{*}_{e_0 e_1} \coloneqq h_{e_0 e_1} \times \R$ be the corresponding admissible strip. For the rest of this paper, we adopt the following notation:
\begin{enumerate}
    \item  We denote the intersection point of two adjacent admissible strips  by \[p^{*}_i \coloneqq\mathcal {S}^{*}_{{e}_{i-1} {e}_i}\cap \mathcal{S}^{*}_{{e}_{i} {e}_{i+1}}\]

    \item For each $i \ge 2$, denote the two \textit{singular geodesics} in the admissible plane $F^{*}_{e_i}$ by
\[
(f^{*}_{i})^{-} \coloneqq F^{*}_{e_i} \cap \mathcal{S}^{*}_{{e}_{i-1} {e}_i} \quad \text{and} \quad (f^{*}_{i})^{+} = F^{*}_{e_i} \cap \mathcal{S}^{*}_{{e}_{i} {e}_{i+1}}
\]
\end{enumerate}

In Section~\ref{sec:mainflatray}, we constructed backward spiral paths that redirect any $\qq$-ray (Type I or Type II) to $\zeta^{*}$ (see Proposition~\ref{Prop:zeta-is-top}. The proof can be adapted to show that if $\alpha^{*}$ is of Type I, then $\zeta^{*}$ can be quasi-redirected to $\alpha^{*}$
\begin{lem}
\label{lem:TypeIequivalentmainflatray}
Let $\alpha^{*}$ be an arbitrary $\qq$--ray of Type I in the flip admissible space $X$. Then $\alpha^{*} \sim \zeta^{*}$. 
\end{lem}
In the rest of this section, we address the case when $\alpha^{*}$ is of Type II.

\subsection*{Excursion}
Following Definition~\ref{defn:typesIandII}, we introduce further refinements to the classification of $\qq$--rays of Type II and their behavior in admissible spaces

  We first establish notation that will be used for the remainder of this section. Let  $\alpha^{*}$ be a $\qq$--ray of Type II\@. Let $e_i=[v_{i-1},v_i]$, where $v_0,v_1,\dots$ is a simplified itinerary of $\alpha^{*}$. Let $t^{*}_i$ be the first time $\alpha^{*}$ intersects $ F^{*}_{e_i}$.

\begin{defn}[Sub-exponential Excursion]  We say that $\alpha^{*}$ has \emph{sub-exponential excursion} with respect to the distance in $T$ if
\[ \lim_{i \to \infty}  \frac {\log |t^{*}_i -t^{*}_{i-1} |}{i} = 0\]

\end{defn}
 Let $\gamma^{*} := \gamma^{*}_{\alpha^{*}}$ be the geodesic ray in the Bass--Serre tree $T$ associated to $\alpha^{*}$. Recall that $p^{*}_i = (f^{*}_{i})^{-} \cap (f^{*}_{i})^{+}$ and $p^{*}_{i+1} = (f^{*}_{i+1})^{-} \cap (f^{*}_{i+1})^{+}$, where $(f^{*}_{i})^{+}$ and $(f^{*}_{i+1})^{-}$ are the two singular geodesics of the admissible strip $\mathcal{S}^{*}_{e_{i} e_{i+1}}$. Also recall that $p^{*}_0 \coloneqq \go$ 

\begin{lem}
\label{lem:upperboundofDelta}
 Assume that the excursion of $\alpha^{*}$ is sub-exponential.  Given a constant $0 < \rho_0 < 1/4$.  Then there exists a constant $C >0$ such that $\Delta^{*}_{i} : = d(p^{*}_i, p^{*}_{i+1}) \le C (1 + \rho_0)^i$.
\end{lem}
\begin{proof}
Since the excursion of $\alpha^{*}$ is sub-exponential, it follows that there exists a constant $C'$  such that 
 \[
 d(\alpha^{*}(t^{*}_i), \alpha(t^{*}_{i+1})) < C'(1+ \rho_0)^i
 \]
We define $$\mathcal{X}^{*} : = X_{v_j} \cup \mathcal{S}^{*}_{e_j e_{j+1}} \cup F^{*}_{e_{j+1}} \cup \mathcal{S}^{*}_{e_{j+1} e_{j+2}} \cup F^{*}_{e_{j+2}} \cup \mathcal{S}^{*}_{e_{j+2} e_{j+3}} \cup F^{*}_{e_{j+3}} \cup \mathcal{S}^{*}_{e_{j+3} e_{j+4}} \cup X_{v_{j+4}}$$
We recall from~\cite[Lemma~4.3] {CK02} that the subspace $\mathcal{X}^{*}$ is a $A_1$--quasiconvex for some uniform constant $A_1$ depending only on the geometry of $X$.

Let $\beta^{*}$ be the shortest geodesic joining $X_{v_j}$ to $X_{v_{j+4}}$. It follows that $\beta^{*} \subset N_{A_1} (\mathcal X)$ because of quasi-convexity. The length of $\beta^{*}$ is necessarily greater than $d(p^{*}_{j+2}, p^{*}_{j+3})$ up to some uniform multiplicative and additive constants, that is
 \[
 \operatorname{Length}(\beta^{*}) > \frac{1}{A_2} d(p^{*}_{j+2}, p^{*}_{j+3}) - A_2
 \]
 for some uniform constant $A_2>0$.
 
 Since $\beta^{*}$ is a shortest geodesic connecting $X_{v_j}$ to $X_{v_{j+4}}$, we have $$d(\alpha^{*}(t^{*}_j), \alpha^{*}(t^{*}_{j+4})) > \operatorname{Length}(\beta^{*})$$  We thus obtain
 \begin{align*}
      \Delta^{*}_{j+2} &= d(p^{*}_{j+2}, p^{*}_{j+3}) \le A_{2}( A_2 + \operatorname{Length}(\beta^{*}) ) \\
      &\le A_{2}^2 + A_2 d (\alpha(t^{*}_j), \alpha(t^{*}_{j+4})) \\
       &\le A_{2}^2 + A_{2} C' \, \sum_{s = j}^{j+4} (1+\rho_0)^s \le C (1+\rho)^{j+2}
 \end{align*}
 for some constant $C  = C(\rho_0, A_2, C')$. The claim is confirmed.    
\end{proof}

\begin{lemma}\label{Thm:Classification} 
Let $\alpha^{*}$ be a $\qq$--ray of Type II\@. 
If the excursion of $\alpha^{*}$ is sub-exponential then $\zeta^{*}$ can not be quasi-redirected to $\alpha^{*}$
\end{lemma} 

\begin{proof}
By way of contradiction, suppose that at every radius $r$, there is always a uniform quasi-geodesic $\xi^{*}$ (say $\xi^{*}$ is a $\qq$--ray for some $\qq = (q, Q)$) that quasi-redirects $\zeta^{*}$ to $\alpha^{*}$ at the radius $r$.
Let $t^{*}_k$ be the first time $\xi^{*}$ visits $F^{*}_{e_k}$ and denote 
\[
\ell^{*}_k \coloneqq d(p^{*}_k, \gamma(t^{*}_k))
\]
Since $\xi^{*}$ is a ${\qq}$--ray, there exists a constant $\rho_0 = \rho_0(q, Q) > 0$ such that 
\begin{equation}
    t^{*}_0 = r \quad \text{and} \quad t^{*}_{k+1} - t^{*}_k \ge \rho_0 \ell^{*}_k 
\end{equation}
Another way to travel from $\gothic o = p^{*}_0$ to $\xi^{*}(t^{*}_k)$ is to go along the path $[p^{*}_0, p^{*}_1], [p^{*}_1,p^{*}_2], \ldots, [p^{*}_{k-1}, p^{*}_k]$ which is a $(\mu, \mu)$--quasi-geodesic where $\mu$ is an uniform constant (see \cite[Proposition~3.8]{NY23} and then go up or down a distance of $\ell^{*}_k$ to reach $\xi^{*}(t^{*}_k)$. 

Recall that $\Delta^{*}_{i} : = d(p^{*}_i, p^{*}_{i+1})$. Again since $\xi^{*}$ is a ${\qq}$--ray we have that
\begin{equation}
\label{equ:lowerboundofell}
    \ell^{*}_k + \sum_{i=0}^{k-1} \Delta^{*}_i \ge \rho_0 t^{*}_k
\end{equation}
Define \[\rho_1 = \rho_{0}^{2}/2\] and pick an arbitrary $0 < \rho < \rho_1$. 

Since the excursion of $\alpha^{*}$ is sub-exponential, it follows from Lemma~\ref{lem:upperboundofDelta} that there exists a constant $C = C(\alpha^{*})>0$ such that for every $i \ge 0$ then 
$\Delta^{*}_{i}  = d(p^{*}_i, p^{*}_{i+1}) \le C (1 + \rho)^i$
and hence \[\sum_{i=0}^{k} \Delta^{*}_i \le C \sum_{i=0}^{k} (1+ \rho)^i \le \frac{C}{\rho} (\rho+1)^k\]

\textbf{Claim:} 
\begin{equation}
    \forall r > 2C/ (\rho \rho_0) \Longrightarrow t^{*}_{k+1} \ge r (1+ \rho_1)^{k+1} \quad \text{and} \quad \ell^{*}_k \ge \frac{r \rho_0}{2} (1 + \rho_1)^{k+1} 
\end{equation}
Indeed, we prove the above claim by induction. The base case is obvious, so we assume the claim is true for all $i \le k$. We have
\begin{align*}
    t^{*}_{k+1} & \ge t^{*}_k + \rho_0 \ell^{*}_k \\
    &\ge r (1+\rho_1)^{k} + \frac{r \rho_{0}^2}{2} (1 + \rho_1)^k 
    \ge r (1 + \rho_1)^k (1 + \frac{\rho_{0}^2}{2}) \\
    &\ge r (1 + \rho_1)^{k+1}
\end{align*}
Using this and (\ref{equ:lowerboundofell}), we have
\begin{align*}
    \ell^{*}_{k+1} &\ge \rho_{0} t^{*}_{k+1} - \sum_{i=0}^{k} \Delta^{*}_i 
    \ge \rho_{0} t^{*}_{k+1}  - \frac{C}{\rho_0} (1 + \rho_0)^{k+1} \\
    &\ge r \rho_0 (1+\rho_1)^{k+1} - \frac{C}{\rho_0} (1+ \rho_1)^{k+1} \\ 
    &= (1+\rho_1)^{k+1} (r \rho_0 - \frac{C}{\rho}) 
    \ge \frac{r \rho_0}{2} (1+ \rho_1)^{k+1}
\end{align*}
On the other hand,  we have 
\[
\sum_{i=0}^{k+1} d(\alpha^{*}(t_i), \alpha^{*}(t_{i-1})) \le \frac{C}{\rho} (1 + \rho)^{k+1} < \frac{C}{\rho} \frac{t^{*}_{k+1}}{r}
\] for $r$ sufficiently large.

In other words, $\xi^{*}$ arrives in $F^{*}_{e_k}$ long after $\alpha^{*}$ has left $F^{*}_{e_k}$. It is a routine computation to shown that for a sufficiently large $r $, we have $d(\xi^{*}(t^{*}_{k+1}), \alpha^{*}(t_{k+1})) \to \infty$. Therefore it is impossible for $\xi^{*}$ to eventually coincide with $\alpha^{*}$.

In conclusion, we have shown that for every ${\qq} = (q, Q)$, there exists a sufficiently large constant $r >0$ such that there is no ${\qq}$--ray $\xi^{*}$ with $\xi^{*}|_r = \zeta^{*}|_r$ and $\xi^{*}$ is eventually equal to $\alpha^{*}$. Therefore $\zeta^{*}$ can not be quasi-redirected to $\alpha^{*}$.
\end{proof}

The following lemma is extracted from the proof of \cite[Proposition~4.2]{QR24}. 

\begin{lem}
    \label{lem:redirecting-in-product}
Let $X = A \times B$ be a product of two proper geodesic metric spaces. Then there exists a pair of constants $\qq = (q, Q) \in [1, \infty) \times [0, \infty)$ such that the following holds. For every four points $x, y,z,t \in X$. Suppose that $d(z,t) > 8  d(x,y)$ then there exists a $(q, Q)$--quasi-geodesic $\gamma$ in $X$ such that $\gamma|_{r} = [x,y]$ where $r:= d(x, y)$ and $\gamma_{+} \in \{z, t\}$.
\end{lem}

\begin{proof}
    Since $d(z,t) > 8 d(x,y) = 8r$, it follows that either $d(x, z) > 4 d(x, y) = 4r$ or $d(x, t) > 4 d(x,y) = 4r$. The proof follows line by line from that of \cite[Proposition~4.2]{QR24}.
\end{proof}

\begin{prop}\label{prop:classifyspecialQR}
    Let $\alpha^{*}$ be an arbitrary $\qq$--ray of Type~II in $X$.  
   If the excursion of $\alpha^{*}$ is not sub-exponential then $\alpha^{*} \sim \zeta^{*}$.
\end{prop}

\begin{proof}
Let $t^{*}_i$ be the first time $\alpha^{*}$ intersects the admissible plane $F^{*}_{e_i}$.
Since $\alpha^{*}$ does not exhibit sub-exponential excursions, there exists a sufficiently small constant $\rho_{0} \in (0,1/16)$ such that for every $r >0$, there exists $k \in \mathbb{Z}_{+}$ satisfying
    \[
    t^{*}_k - t^{*}_{k-1} \geq r(1+ \rho_0)^k.
    \]
Define $k \in \mathbb{Z}_{+}$ to be the smallest integer such that
\[
\begin{cases}
    t^{*}_k - t^{*}_{k-1} &\geq r (1 + \rho_0)^k, \\
    t^{*}_i - t^{*}_{i-1} &< r (1 +\rho_0)^i, \quad \forall 1 \leq i \leq k-1.
\end{cases}
\]

Consider the geodesic path $e_1 \cdot e_2 \cdots e_{k+1}$ in the Bass-Serre tree $T$. Let $\mathcal{T}$ be the standard template associated with this geodesic segment, as defined in Definition~\ref{defn:standardtemplate}. Denote $$\phi \colon \mathcal{T} \to X$$ be the $(L,A)$--quasi-isometric embedding given by Lemma~\ref{lem:CK02}. Define $x^{*}_{k} = \alpha(t^{*}_k)$ and $x^{*}_{k-1} = \alpha(t^{*}_{k-1})$, and let $x_k, x_{k-1} \in \mathcal{T}$ such that $\phi(x_k) = x^{*}_k$ and $\phi(x_{k-1}) = x^{*}_{k-1}$. 

We recall that $\phi$ maps planes $F_i$ and strips $S_{i(i+1)}$ of $\mathcal{T}$ to the $K$--neighborhood (where $K$ depends only on the geometry of $X$) of the admissible planes $F^{*}_{e_i}$ and the admissible strips $\mathcal{S}^{*}_{e_i e_{i+1}}$ in $X$. Consequently, $\phi$ maps $p_i$ to a $K$--neighborhood of $p^{*}_i$. Thus, up to uniform errors, we assume $x_k \in F_{k}$ and $x_{k-1} \in F_{k-1}$.

Within the template $\mathcal{T}$, and with respect to the sequence $t_1, t_2, \ldots, t_k$, we define a forward spiral path of Type II, denoted by $J_{r, k}$, as introduced in Section~\ref{subsection:forward-spiral-II}. As discussed in Section~\ref{subsection:forward-spiral-II}, this path satisfies the following properties:
\begin{enumerate}
    \item It is a $(\nu, \nu)$--quasi-geodesic for some $\nu$ independent of $r$.
    \item $d(z_{k-1}, p_{k-1})/ \rho_0 < t_k -t_{k-1} \sim d(x_k , x_{k-1})$.
\end{enumerate}

Applying Lemma~\ref{lem:redirecting-in-product} to the four points $p_{k-1}, z_{k-1}, x_k, x_{k-1}$ in the product space $F_{k-1} \cup S_{k-1, k} \cup F_k = (l_{k-1} \cup w_{k-1, k} \cup l_k) \times \mathbb{R}$ which is a subspace of the template $\mathcal{T}$, we obtain a $(q, Q)$--quasi-geodesic $\gamma$ in $F_{k-1} \cup S_{k-1, k} \cup F_k$ such that $\gamma|_{s} = [z_{k-1} , q_{k-1}]$, where $s := d(z_{k-1}, q_{k-1})$, and $\gamma_{+} \in \{x_k, x_{k-1}\}$. 

Since $[y_{k-1}, z_{k-1}]$ is the shortest segment realizing the distance between the two sets $J_{r,k} \backslash [y_{k-1}, z_{k-1}]$ and $\gamma$, it follows that $J_{r, k} \cup \gamma$ is a $(q', Q')$--quasi-geodesic for some $q' = q'(q, Q)$ and $Q' = Q'(q, Q)$. 

Defining $L^{*} := \phi ( J_{r,k} \cup \gamma) \subset X$, we conclude that it is a $(q'', Q'')$--quasi-geodesic for some $q'' = q'' (q, Q, L, A)$ and $Q'' = Q'' (q, Q, L, A)$. Moreover, this quasi-geodesic satisfies $L^{*}|_{r} = \zeta^{*}|_{r}$ and $(L^{*})_{+} \in \{x^{*}_k, x^{*}_{k-1} \} \subset \alpha^{*}$. 

Applying the Segment-to-quasi-geodesic ray Surgery (see Lemma~\ref{item_surgery_segment}), we conclude that $\zeta^{*}$ can be quasi-redirected to $\alpha^{*}$ at radius $r$. Since this holds for every $r>0$, it follows that $\zeta^{*} \preceq \alpha^{*}$. By Proposition~\ref{Prop:zeta-is-top}, we obtain $\alpha^{*} \preceq \zeta^{*}$. Therefore, $\zeta^{*} \sim \alpha^{*}$.
\end{proof}

\begin{rem}
The idea of the above proof is that
we can transition from $J_{r,k}$ to $\alpha^{*}|_{\geq R}$ for sufficiently large $R$, provided that a buffer region exists between them. This buffer must have a product structure and a thickness that grows linearly with $r$, ensuring sufficient space for a smooth transition or ``landing" between the two paths.
\end{rem}

\begin{prop}\label{prop:differentitinerariesnotthesame}
    Let $\alpha^{*}$ and $\alpha'^{*}$ be two $\qq$--rays of Type~II in $X$ with different simplified itineraries and with sub-exponential excursions. Then $\alpha^{*}$ can not be quasi-redirected to $\alpha'^{*}$ and vice versa. 
\end{prop}
\begin{proof}
   By way of contradiction, suppose that $[\alpha^{*}] = [\alpha'^{*}]$. In particular, we have $\alpha'^{*} \preceq \alpha^{*}$.

   \textit{Claim: $\zeta^{*} \preceq \alpha^{*}$}.
   
      Indeed, by Lemma~\ref{lem:upperboundofDelta}, for a sufficiently small constant $\rho_0$, there exists $C >0$ such that $$\Delta^{*}_{i} : = d(p^{*}_{i}, p^{*}_{i+1}) \le C (1+\rho_0)^i$$
      Let $\gamma_{\alpha^{*}} = e_1 \cdot e_2 \cdots$ be the simplified itinerary associated to $\alpha^{*}$ as defined in Definition~\ref{defn:typesIandII}.
Let $\mathcal{T}$ be the standard template associated with this geodesic segment, as defined in Definition~\ref{defn:standardtemplate}. Denote by $\phi \colon \mathcal{T} \to X$ the $(L,A)$--quasi-isometric embedding given by Lemma~\ref{lem:CK02}.
      We recall that $\phi$ maps planes and strips of $\mathcal{T}$ to the $K$--neighborhood (where $K$ depends only on the geometry of $X$) of the planes $F^{*}_{e_i}$ and the strips $\mathcal{S}^{*}_{e_i e_{i+1}}$ in $X$. Consequently, $\phi$ maps $p_i$ to a $K$--neighborhood of $p^{*}_i$, and hence
      $$w_i : = d(p_i, p_{i+1}) \le C' (1+\rho_0)^i$$ for some $C' = C'(C, L,A)$.
      
Let $r>0$ be a sufficiently large constant. With respect to the above data, let $\gamma : = L_{r,k}$ be the forward spiral path of Type I in the template $\mathcal{T}$ constructed in Section~\ref{subsection:logspiralTypeI}   such that $\left.\gamma\right|_r=\left.\zeta\right|_r$.
We then define $$\gamma^{*} : = \phi(\gamma)$$
Let $t^{*}_k$ be the first time $\gamma^{*}\left(t^{*}_k\right) \in F^{*}_{e_k}$ and denote \[\ell^{*}_k\coloneqq d\left(\gamma^{*}\left(t^{*}_k\right), p^{*}_k\right)\]

Now choose $R \gg \ell_k$  we  consider a quasi-geodesic $\beta^{*}$ quasi-redirecting $\alpha^{*}$ to $\alpha'^{*}$ at radius $R$. Such a $\beta^{*}$ exists since $\alpha^{*} \preceq \alpha'^{*}$. Recall that $\alpha^{*}$ and $\alpha'^{*}$ have different simplified itineraries.
Then $\beta^{*}$ arrives at and leaves $F^{*}_{e_k}$ much later than $\gamma^{*}$. Hence, similar arguments as in the proof of Proposition~\ref{prop:classifyspecialQR}, we can redirect $\gamma^{*}$ to $\beta^{*}$, that is, construct a quasi-geodesic ray $\gamma'^{*}$ where $\gamma^{*}[0, t^{*}_k]=\gamma'^{*}\left[0, t^{*}_k\right]$, and $\gamma'^{*}$ is eventually equal to $\beta^{*}$.

Since $\beta^{*}$ is eventually equal to $\alpha^{*}$ it implies that $\gamma'^{*}$ is eventually equal to $\alpha^{*}$, and thus $\gamma'^{*}$ quasi-redirects $\zeta^{*}$ to $\alpha^{*}$ at radius $r$. This can be done for every $r$ with uniform constants. Hence $\zeta^{*} \preceq \alpha^{*}$. This would contradict Lemma~\ref{Thm:Classification}.
\end{proof}

\begin{prop}\label{prop:important}
Let $\alpha^{*}$ be a $\qq$-ray that is of Type II and is sub-exponential. Then there exists a geodesic ray $\alpha^{*}_0$ in $X$ whose simplified itinerary is the sequence $\gamma^{*}_{\alpha^{*}}$ such that $\alpha^{*} \sim \alpha^{*}_{0}$.
\end{prop}

\begin{proof}
    Choose a sequence $r_i \to \infty$ and let $x_i$ be starting point of the 
quasi-geodesic $\alpha^{*}|_{\geq r_i}$. Then $x_i$ is also the closest point in $\alpha^{*}|_{\geq r_i}$
to $\go$. Let
\[
\alpha^{*}_i = [\go, x_i] \cup \alpha^{*}|_{\geq r_i}. 
\]
By Lemma~\ref{concate}, $\alpha^{*}_i$ is a $(3q, Q)$--quasi-geodesic ray. Up to taking a subsequence, 
the geodesic segments $[\go, x_i]$ converge to a geodesic ray $\alpha^{*}_0$. It is shown in \cite[Lemma~3.5]{QR24} that $\alpha^{*}_0 \preceq \alpha^{*}$ with quasi-redirecting constants $(3q, Q+1)$. We thus only need to show that $\alpha^{*} \preceq \alpha^{*}_{0}$. 

Let $t^{*}_i$ be the first time $\alpha^{*}$ visits $F^{*}_{e_i}$ and let $q^{*}_i$ be the first time $\alpha^{*}_0$ visits $F^{*}_{e_i}$. 

Suppose that $$d(\alpha^{*}_{0}(q^{*}_i), \alpha^{*}(t^{*}_i)) > \frac{d(\go, q^{*}_i)}{2}$$ for every $i$. Since $\alpha^{*}$ is a $(q, Q)$--quasi-geodesic, there exists a constant $\rho_0 >0, q^{*}>0$ depending on $\qq$ so that for $k$ large enough then
\[
t^{*}_{k+1} -t^{*}_k > \rho_{0} t^{*}_{k} + q \delta_k  
\]

Let $\mathcal{T}$ be the standard template associated with $\gamma^{*}_{\alpha^{*}_0}$, as defined in Definition~\ref{defn:standardtemplate}. Denote by $\phi \colon \mathcal{T} \to X$ the $(L,A)$--quasi-isometric embedding given by Lemma~\ref{lem:CK02}.

Pick a constant $\rho > 0$ sufficiently small. According to Proposition~\ref{prop:smallslope}, we can construct a quasi-geodesic $L_{r,k+2}$ in the standard template $\mathcal{T}$  with \[
l_{k+1} - l_k < \rho l_k + \delta_k
\] where $l_i : = \operatorname{Length} (L_{r,i})$.

As usual, we define $L^{*}_{r, k+2} = \phi (L_{r, k+2})$ is  a path in $X$.

Since the sequence $\{l_k\}$ grows more slowly than $\{t^{*}_k\}$. That is to say for sufficiently large $k$, $\alpha^{*}$ arrives at and leaves $F^{*}_{e_k}$ much later than $L^{*}_{r,k}$.  Hence, similar arguments as in the proof of Proposition~\ref{prop:classifyspecialQR}, we can redirect $\zeta^{*}$ to $\alpha^{*}$ at radius $r$ as $\zeta^{*}$ is identical with $L^{*}_{r,k}$ up to radius $r$. Therefore $\zeta^{*} \preceq \alpha^{*}$. But this would contradict Lemma~\ref{Thm:Classification}.

Therefore there must be a sequence of $i \to \infty$ such that $$d(\alpha^{*}_{0}(q^{*}_i), \alpha^{*}(t^{*}_i)) \le \frac{d(\go, q^{*}_i)}{2}$$ By Surgery Lemma~\ref{redirect11} we obtain that $\alpha^{*} \preceq \alpha^{*}_0$ with redirecting constant $(9q, Q)$.  Therefore $ \alpha^{*} \sim \alpha^{*}_0$.
\end{proof}

Now we have enough ingredients to claim the existence of the QR-boundary of $X$.
\begin{thm}\label{mainwithproof}
The quasi-redirecting boundary $\partial X$ exists and is non-Hausdorff.
\end{thm}
\begin{proof}
By~\cite[Lemma 2.3]{QR24}, all finitely generated groups satisfy QR Assumption 0. Here we check QR-Assumptions $1$ and 2. That is, for every $\bfa \in P(X)$, there is a geodesic representative, and there is a function 
\[
f_\bfa : \, [1, \infty) \times [0, \infty) \to [1, \infty) \times [0, \infty), 
\] 
 any ${\qq}$--ray $\alpha^{*} \in \mathbf a$ can be
$f_{\mathbf a}({\qq})$--quasi-redirected to the representative of $\mathbf{a}$.

If $\alpha^{*}$ is of Type I or of Type II, but it does not have sub-exponential excursion, then by Proposition~\ref{Prop:zeta-is-top}  $\alpha^{*} \preceq \zeta^{*}$ with constants $\qq'(\qq)$. If otherwise, Proposition~\ref{prop:classifyspecialQR} show that $\zeta^{*} \preceq \alpha^{*}$ with uniform constants. Thus $\zeta^{*}$ is a suitable geodesic representative of $[\alpha^{*}]$ and $f_{[\alpha^{*}]} = \qq'(\qq)$. 
Otherwise, $\alpha^{*}$ is of Type II and sub-exponential, then Proposition~\ref{prop:important} shows that $\alpha^{*}_0$ is a geodesic representative of $[\alpha^{*}]$ and the redirecting function is $f_{[\alpha^{*}]} = (9q, Q)$. Thus $X$ satisfies all three QR-Assumption $0, 1,2$, and $\partial X$ is well-defined and QI-invariant.

\end{proof} 

\subsection{Proof of Theorem~\ref{thm:introQRofadmissibleCAT0}}
Kapovich and Leeb demonstrate that for every graph manifold $ M $, there exists a flip graph manifold $ N $ whose fundamental groups are quasi-isometric \cite{KL98}. This result is further extended to Croke-Kleiner admissible groups in \cite[Theorem 1.6]{Ngu25}, which specifically establishes the existence of a flip admissible group $ G' $ (that is a Croke-Kleiner  admissible group acts geometrically on a flip admissible space) such that $ G $ and $ G' $ are quasi-isometric. Noting that the quasi-redirecting boundary is a quasi-isometric invariant, the conclusion then follows from Theorem~\ref{mainwithproof}.

\section{Application to 3-manifold groups and certain right-angled coxeter groups}\label{sec:3-manifolds}

\subsection{QR boundary of 3-manifold groups}

Now we are ready to prove Theorem~\ref{thm:3-manifold}.

\begin{proof}[Proof of Theorem~\ref{thm:3-manifold} ]
 Let $M$ be a non-geometric 3-manifold.  Then $M$ is either a graph manifold or a mixed 3-manifold.

 \textit{Case~1:} $M$ is a graph manifold. At first, by passing to a finite cover $M'$ of $M$, we
can assume that each Seifert piece $M_v$ of $M$ is a product $S_v \times S^1$ where $S_v$ is a hyperbolic surface with nonempty boundary \cite{KL98}. This is allowable since quasi-redirecting boundary is a quasi-isometric invariant. The fundamental group $\pi_1(S_v)$ is free, and hence it is omnipotent. Therefore $\pi_1(M)$ is a Croke-Kleiner  admissible group where each vertex group $\pi_1(M_v)$ is a CAT(0) group and its quotient $\pi_1(S_v)$ is omnipotent. Theorem~\ref{thm:introQRofadmissibleCAT0} implies that $\pi_1(M)$ satisfies all three QR-Assumptions and $\partial \pi_1(M)$ is well-defined.

 \textit{Case~2:} $M$ is a mixed 3-manifold.  Let $M_1, \dots, M_k$ be the maximal graph manifold components and Seifert fibered pieces of the torus decomposition of $M$. Let $S_1, \dots, S_{\ell}$ be the tori in the boundary of $M$ that bound a hyperbolic piece, and let $T_1, \dots,T_m$ be the tori in the torus decomposition of $M$ that separate two hyperbolic components. According to~\cite{Dah03} (see also~\cite{BW13}), $\pi_1(M)$ is hyperbolic relative to 
\[
\mathbb{P} = \{\pi_1(M_p)\}_{p=1}^k \cup \{\pi_1(S_q)\}_{q=1}^\ell \cup \{\pi_1(T_r)\}_{r=1}^m.
\]
We note that the quasi-redirecting boundaries of 
$\pi_1(S_q)$, $\pi_1(T_r)$ exist since they are isomorphic to $\Z^2$. Case~1 implies the existence of the quasi-redirecting boundary of $\pi_1(M_p)$. Thus, we apply  Theorem~\ref{thm:introRHGs} to conclude that the quasi-redirecting boundary of $\pi_1(M)$ exists. 
\end{proof}

\subsection{QR boundary of certain right-angled Coxeter groups}
\label{sec:QR-RACG}

 Given a graph $\Gamma$, define $\Gamma^4$ as the graph whose vertices are induced 4--cycles of $\Gamma$. Two vertices in $\Gamma^4$ are adjacent if and only if the corresponding induced 4-cycles in $\Gamma$ have two nonadjacent vertices in common. 

\begin{defn}[Constructed from squares]\label{defn:CFS}
A graph $\Gamma$ is \emph{$\mathcal{CFS}$} if $\Gamma$ is the join $\Omega*K$ where $K$ is a (possibly empty) clique and $\Omega$ is a non-empty subgraph such that $\Omega^4$ has a connected component $T$ such that every vertex of $\Omega$ is contained in a $4$--cycle that is a vertex of $T$. If $\Gamma$ is $\mathcal{CFS}$, then we will say that the right-angled Coxeter group $W_\Gamma$ is $\mathcal{CFS}$. 
\end{defn}

\subsection*{Standing Assumptions}
The planar flag complex $\Delta\subset \field{S}^2$:
\begin{enumerate}
    \item is connected with no separating vertices and no separating edges ($W_\Delta$ is one-ended);
    \item contains at least one induced $4$-cycle ($W_\Delta$ is not hyperbolic);
    \item is not a $4$-cycle and not a cone of a $4$-cycle ($G_\Delta$ is not virtually $\Z^2$).
\end{enumerate}
 
\begin{prop}\label{prop:RACG} 
    Let $\Delta\subset \field{S}^2$ be a flag complex satisfying Standing Assumptions. 
    Assume that either $\Delta=\field{S}^2$ or the boundary of each region in $\field{S}^2-\Delta$ is a $4$--cycle. 
    Then the quasi-redirecting boundary of the right-angled Coxeter groups $W_\Gamma$ exists. 
\end{prop}

\begin{proof}
    It is shown in~\cite[Theorem~1.1]{NT19} and~\cite{HNT20} that there are mutually exclusive cases as bellow:

    (1): If $\Delta$ is a suspension of some $n$-cycle ($n\geq 4$) or some broken line (i.e.\ a finite disjoint union of vertices and finite trees with vertex degrees $1$ or $2$), then $G$ contains a finite index subgroup $G'$ which is isomorphic to $\pi_1(M)$ with $M$ is a Seifert manifold. In this case, there is a finite cover $M' \to M$  such that $M' = F \times S^1$ where $F$ is a hyperbolic surface with a nonempty boundary, and thus $\partial (\pi_1(M'))$ consists only one point by Proposition~\ref{prop:product}. Since $G$ is quasi-isometric to $\pi_1(M')$, it follows from Theorem~\ref{thm:collectfacts} that $\partial G$ consists only one point.

    (2): If the $1$-skeleton of $\Delta$ is $\mathcal{CFS}$ and does not satisfy (1) then $G$ contains a finite index subgroup $G'$ which is isomorphic to $\pi_1(M)$ with $M$ is a graph manifold. If the 1-skeleton of $\Delta$ contains a separating induced $4$-cycle and is not $\mathcal{CFS}$, then $M$ is a mixed manifold. In these two cases, it follows from Theorem~\ref{thm:3-manifold} that the quasi-redirecting boundary of $\pi_1 (M)$ exists, and so does $G$.

    (3): If the $1$-skeleton of $\Delta$ has no separating induced $4$-cycle and is not $\mathcal{CFS}$, then $G$ contains a finite index subgroup $G'$ which is isomorphic to $\pi_1(M)$ with $M$ is a hyperbolic 3-manifold with tori boundary. In this case, $\pi_1(M)$ is hyperbolic relative to a finite collection of $\Z^2$ which have trivial QR-boundaries, and Theorem~\ref{thm:introRHGs} implies the existence of the quasi-redirecting boundary of $\pi_1(M)$, and so does $G$.
\end{proof}

\begin{thm}\label{maincox}
  Let $\Gamma$ be a graph whose flag complex $\Delta$ is planar.  Then the quasi-redirecting boundary of the right-angled Coxeter group $W_\Gamma$ exists.
\end{thm}

\begin{proof}
    According to~\cite[Theorem~1.2]{HNT20}, there is a collection $\mathbb J$ of $\mathcal{CFS}$ subgraphs of $\Gamma$ such that the right-angled Coxeter group $G_\Gamma$ is relatively hyperbolic with respect to the collection $\mathbb P=\set{G_J}{J\in \mathbb J}$. 
    By Proposition~\ref{prop:RACG}, the quasi-redirecting of each peripheral subgroup $G_J \in \mathbb P$. We now apply Theorem~\ref{thm:introRHGs} to obtain the conclusion.
\end{proof}

\bibliographystyle{alpha}


\begin{thebibliography}{AKB12} 
\bibitem[ANR23]{CarolynNguenRamussen2024} {C.~Abbott, H.~Nguyen, A.~Rasmussen}. \textit{Largest hyperbolic actions
of 3–manifold groups}. Bull. London Math. Soc., 56(10):3090–3113, 2024.

\bibitem[Baj07]{Baj07} {J.~Bajpai}. \textit{Omnipotence of surface groups}. Master’s thesis, McGill University, 2007.

\bibitem[BH99]{BH99} {M.~Bridson, A.~Haefliger}. \textit{Metric spaces of non-positive curvature}, volume 319 of
Grundlehren der mathematischen Wissenschaften. Springer-Verlag, Berlin, 1999.

\bibitem[Bow12]{Bow12} {B.~Bowditch}. \textit{Relatively hyperbolic groups}. Internat. J. Algebra Comput., 22(3):1250016, 66, 2012.

\bibitem[BW13]{BW13} {H.~Bigdely, D.~Wise}. \textit{Quasiconvexity and relatively hyperbolic groups that split}. Michigan Math. J., 62(2):387–406, 2013.

\bibitem[CDG22]{CDG22} {M.~Cordes, M.~Dussaule, I.~Gekhtman}. \textit{An embedding of the Morse boundary in the Martin boundary}. Algebr. Geom. Topol., 22(3):1217–1253, 2022.

\bibitem[CK00]{CK00} {C.~Croke, B.~Kleiner}. \textit{Spaces with nonpositive curvature and their ideal boundaries}. Topology., 39(3):549–556, 2000.

\bibitem[CK02]{CK02} {C.~Croke, B.~Kleiner}. \textit{The geodesic flow of a nonpositively curved graph manifold}. Geom. Funct. Anal., 12(3):479–545, 2002.


\bibitem[Cor17]{Cor17} {M.~Cordes}. \textit{Morse boundaries of proper geodesic metric spaces}. Groups Geom. Dyn., 11(4):1281–1306, 2017.

\bibitem[CS15]{CS15} {R.~Charney, H.~Sultan}. \textit{Contracting boundaries of CAT(0) spaces}. J. Topol., 8(1):93–117, 2015.

\bibitem[Dah03]{Dah03} {F.~Dahmani}. \textit{Combination of convergence groups}. Geom. Topol., 7:933–963, 2003.

\bibitem[DS05]{DS05} {C.~Drutu, M.~Sapir}. \textit{Tree-graded spaces and asymptotic cones of groups}. Topology., 44(5):959–1058, 2005. With an appendix by Denis Osin and Mark Sapir.


\bibitem[Far98]{Far98} {B.~Farb}. \textit{Relatively hyperbolic groups}. Geom. Funct. Anal., 8(5):810–840, 1998.

\bibitem[GQV24]{GQV24} {J.~Garcia, Y.~Qing, E.~Vest}. \textit{Topological and Dynamic Properties of the Sublinearly
Morse boundary and the Quasi-Redirecting Boundary}. \url{https://arxiv.org/abs/2408.10105}.

\bibitem[Ger94]{Ger94} {S.~Gersten}. \textit{Quadratic divergence of geodesics in CAT(0) spaces}. Geometric \& Functional Analysis GAFA 4(1) (1994): 37-51.

\bibitem[Gro87]{Gro87} {M.~Gromov}. \textit{Hyperbolic groups}. In Essays in group theory, volume 8 of Math. Sci. Res. Inst.
Publ., pages 75–263. Springer, New York, 1987.

\bibitem[HNT20]{HNT20} {M~Haulmark, H.~Nguyen, H.~Tran}. \textit{The relative hyperbolicity
and manifold structure of certain right-angled Coxeter groups}. Internat. J. Algebra Comput.,
30(3):501–537, 2020.

\bibitem[HNY23]{HNY23} {S.~Han, H.~Nguyen, W.~Yang}. \textit{Property (QT) for 3-manifold groups}. To appear in Algebr. Geom. Topol. \url{https://arxiv.org/abs/2108.03361v4}.

\bibitem[HRSS24]{HRSS24} {M~Hagen, J.~Russell, A.~Sisto, D.~Spriano}. \textit{Equivariant hierarchically hyperbolic structures for 3-manifold groups via quasimorphisms}. To appear in Ann. Inst. Fourier. \url{https://arxiv.org/abs/2206.12244}.

\bibitem[Hru10]{Hru10} {G.~Hruska}. \textit{Relative hyperbolicity and relative quasiconvexity for countable groups}. Algebr. Geom. Topol., 10(3):1807–1856, 2010.

\bibitem[KL98]{KL98} {M.~Kapovich, B.~Leeb}. \textit{3-manifold groups and nonpositive curvature}. Geom. Funct. Anal., 8(5):841–852, 1998.


\bibitem[MN24]{MN24} {A.~Margolis, H.~Nguyen}. \textit{Quasi-isometric rigidity of extended admissible groups}. \url{https://arxiv.org/abs/2401.03635}.


\bibitem[Ngu25]{Ngu25} {H.~Nguyen}. \textit{An extension of Kapovich-Leeb's theorem}. \url{https://drive.google.com/file/d/1KmYZL0j6ergoCYrH49blbaAN366kUf6i/view}.

\bibitem[NQ24]{NQ24} {H.~Nguyen, Y.~Qing}. \textit{Sublinearly Morse boundary of CAT(0) admissible groups}. J. Group Theory., 27(4):857–897, 2024.

\bibitem[NT19]{NT19} {H.~Nguyen, H.~Tran}. \textit{On the coarse geometry of certain right-angled Coxeter groups}. Algebr. Geom. Topol., 19(6):3075–3118, 2019.

\bibitem[NY23]{NY23} {H.~Nguyen, W.~Yang}. \textit{Croke-Kleiner admissible groups: property (QT) and quasiconvexity}. Michigan Math. J., 73(5):971–1019, 2023.


\bibitem[QR24]{QR24} {Y.~Qing, K.~Rafi}. \textit{The quasi-redirecting Boundary}. \url{https://arxiv.org/abs/2406.16794}.

\bibitem[QRT22]{QRT22} {Y.~Qing, K.~Rafi, G.~Tiozzo}. \textit{Sublinearly Morse boundary I: CAT (0) spaces}. Adv.Math., 404(part B):Paper No. 108442, 51, 2022.

\bibitem[QRT24]{QRT24} {Y.~Qing, K.~Rafi, G.~Tiozzo}. \textit{Sublinearly Morse boundary, II: Proper geodesic spaces}. Geom. Topol., 28(4):1829–1889, 2024.

\bibitem[QY24]{QY24} {Y.~Qing, W.~Yang}. \textit{Genericity of sublinearly Morse directions in general metric spaces}. \url{https://arxiv.org/abs/2404.18762}.

\bibitem[Sis12]{Sis12} {A.~Sisto}. \textit{On metric relative hyperbolicity}. \url{https://arxiv.org/abs/1210.8081}.

\bibitem[Sis13]{Sis13} {A.~Sisto}. \textit{Projections and relative hyperbolicity}. Enseign. Math. (2), 59(1-2):165–181, 2013.

\bibitem[SW79]{SW79} {P.~Scott, T.~Wall}. \textit{Topological methods in group theory}. n Homological group theory (Proc. Sympos., Durham, 1977), volume 36 of London Math. Soc. Lecture Note Ser., pages 137–203. Cambridge Univ. Press, Cambridge-New York, 1979.


\bibitem[SZ24]{SZ24} {A.~Sisto, S.~Zbinden}. \textit{ Nearly-linear solution to the word problem for 3-manifold groups}. \url{https://arxiv.org/abs/2407.18029}.

\bibitem[Tao25]{Tao25} {B.~Tao}. \textit{Property (QT) of relatively hierarchically hyperbolic groups}. \url{https://arxiv.org/abs/2412.20065}.

\bibitem[Wil10]{Wil10} {H.~Wilton}. \textit{irtual retractions, conjugacy separability and omnipotence}. . Algebra, 323(2):323– 335, 2010.

\bibitem[Wis00]{Wis00} {D.~Wise}. \textit{Subgroup separability of graphs of free groups with cyclic edge groups}. Q. J. Math., 51(1):107–129, 2000.

\end{thebibliography}
\end{document}